\documentclass[suppldata]{interact}

\usepackage{epstopdf}
\usepackage[caption=false]{subfig}
\usepackage{color}

\usepackage[numbers,sort&compress]{natbib}
\bibpunct[, ]{[}{]}{,}{n}{,}{,}
\renewcommand\bibfont{\fontsize{10}{12}\selectfont}
\makeatletter
\def\NAT@def@citea{\def\@citea{\NAT@separator}}
\makeatother

\theoremstyle{plain}
\newtheorem{theorem}{Theorem}[section]
\newtheorem{lemma}[theorem]{Lemma}
\newtheorem{corollary}[theorem]{Corollary}
\newtheorem{proposition}[theorem]{Proposition}

\theoremstyle{definition}
\newtheorem{definition}[theorem]{Definition}
\newtheorem{example}[theorem]{Example}

\theoremstyle{remark}
\newtheorem{remark}{Remark}

\begin{document}


\title{The $C$-Numerical Range in Infinite Dimensions}

\author{
\name{Gunther Dirr\textsuperscript{a} and Frederik vom Ende\textsuperscript{b}\thanks{CONTACT Gunther Dirr. Email: dirr@mathematik.uni-wuerzburg.de, Frederik vom Ende (corresponding author). Email: frederik.vom-ende@tum.de}}
\affil{\textsuperscript{a}Department of Mathematics, University of W{\"u}rzburg, 97074 W{\"u}rzburg, Germany\\
\textsuperscript{b}Department of Chemistry, TU Munich, 85747 Garching, Germany}
}

\maketitle

\begin{abstract}
In infinite dimensions and on the level of trace-class operators $C$ rather than matrices, we show that the closure of the $C$-numerical range $W_C(T)$ is always star-shaped with respect to the set $\operatorname{tr}(C)W_e(T)$, where $W_e(T)$ denotes the essential numerical range of the bounded operator $T$. Moreover, the closure of $W_C(T)$ is convex if either $C$ is normal with collinear eigenvalues or if $T$ is essentially self-adjoint. In the case of compact normal operators, the $C$-spectrum of $T$ is a subset of the $C$-numerical range, which itself is a subset of the convex hull of the closure of the $C$-spectrum. This convex hull coincides with the closure of the $C$-numerical range if, in addition, the eigenvalues of $C$ or $T$ are collinear.
\end{abstract}

\begin{amscode}
47A12, 15A60
\end{amscode}

\section{Introduction}
The $C$-numerical range has significant impact on quantum control and quantum information theory 
since the expression $\operatorname{tr}(\rho A)$ can be interpreted as the expectation
value of an observable $A$ with respect to the state $\rho$, that is to say, as the expectation value of a measurement $A$
taken on a quantum system in state $\rho$. 
While in standard quantum mechanics $A$ is self-adjoint and $\rho$ is a (trace-class) density operator, 
there are in fact important applications where $A$ or $\rho$ (or both) are allowed to be non-self-adjoint. Maximizing the absolute value \cite{vonNeumann} or the real part of
$\operatorname{tr}(\rho U^\dagger AU)$ over the unitary orbit of $A$ relate to different optimization problems in the Euclidean geometry of the $C$-numerical range
\cite{article_tsh,TSH_Optimization_QI}.

In the finite-dimensional case, where $A$ and $C$ are assumed to be complex $n \times n$ matrices,
the $C$-numerical range of $A$ is defined by
\begin{align}\label{c_num_range_findim}
W_C(A)=\lbrace\operatorname{tr}(CU^\dagger AU)\,|\,U\in\mathbb C^{n\times n}\text{ unitary}\rbrace\,.
\end{align}
Originally, it was introduced in \cite{GoldbergStraus} as a generalization of the $c$-numerical
range \cite{article_westwick} and the classical numerical range \cite{hausdorff,toeplitz}. Important
properties of the $C$-numerical range are convexity if $C$ is normal with collinear eigenvalues
\cite{article_westwick,article_poon} and star-shapedness with respect to 
$(\operatorname{tr}(C)\operatorname{tr}(A)/n)$ for arbitrary complex $C$, cf.~\cite{article_cheungtsing}. 
For a comprehensive survey, we refer to \cite{article_li_radii}.\medskip

In this work, let $\mathcal H$ be an infinite-dimensional separable complex Hilbert space, $C$ some
trace-class operator on $\mathcal H$, and $T$ some bounded linear operator on $\mathcal H$.
Thus one may introduce the $C$-numerical range of $T$ as follows
\begin{align}\label{eq:WCA-inf-dim}
W_C (T)=\lbrace \operatorname{tr}(CU^\dagger TU)\,|\,U\in\mathcal B(\mathcal H)\text{ unitary}\rbrace\,,
\end{align}
where $\mathcal B(\mathcal H)$ denotes the set of all bounded linear operators acting on $\mathcal H$. 
Clearly, this is a generalization of the finite-dimensional case. Here we take advantage of the fact that
the set of all trace-class operators is a two-sided ideal in the $C^*$-algebra $\mathcal B(\mathcal H)$.
In this setting, however, symmetry in $C$ and $T$ is lost. If one wants to preserve symmetry one could
choose $C$ and $T$ to be Hilbert-Schmidt operators, a direction not pursued in this paper.\medskip

The goal of this paper is to carry over star-shapedness 
or convexity of $W_C(A)$ to the infinite-dimensional setting. Interim results on this subject were achieved by Westwick \cite{article_westwick} and Hughes \cite{article_hughes} for the $c$-numercal range and by Jones \cite{article_jones} for the
$C$-numerical range. Jones, however, pursued a different approach in \cite{article_jones}.
For $C\in\mathbb C^{k\times k}$ and $T\in\mathcal B(\mathcal H)$ he introduced the set
\begin{align}\label{eq:jones_1}
\Big\lbrace\sum\nolimits_{i,j=1}^kc_{ij}\langle f_j,Tf_i\rangle\,\Big|\,\lbrace f_1,\ldots,f_k\rbrace\text{ is orthonormal system in }\mathcal H\Big\rbrace
\end{align}
as the $C$-numerical range of $T$, where $\mathcal H$ can be any infinite-dimensional complex Hilbert
space, and proved that its closure is star-shaped. In doing so,
the essential numerical range $W_e(T)$, or more precisely, the set $\operatorname{tr}(C)W_e(T)$ turned out to be an appropriate replacement of the finite
dimensional star-center $(\operatorname{tr}(C)\operatorname{tr}(A)/n)$. The definition and basic properties of $W_e(T)$ are given, e.g. in \cite{bonsallduncan}.\medskip

This work is organized as follows: After some preliminaries on trace-class operators and set convergence,
we present our main results in Section \ref{sec:results}: (i) Star-shapedness and convexity of the closure of the $C$-numerical range \eqref{eq:WCA-inf-dim} are proved. (ii) A new characterization of $W_e(T)$
is derived which explains the role of $\operatorname{tr}(C)W_e(T)$ as set of star points.
(iii) Some results on the $C$-spectrum in infinite dimensions generalizing
well-known results for matrices \cite{article_marcus, article_sunder} are obtained and consequences for
the $C$-numerical range of compact normal operators $T$ are derived.

\section{Notation and Preliminaries}\label{sec:prelim}
Unless stated otherwise, here and henceforth $\mathcal X$ and $\mathcal Y$ are arbitrary 
infinite-dimensional complex Hilbert spaces while $\mathcal H$ and $\mathcal G$ are reserved
for infinite-dimensional separable complex Hilbert spaces (for short i.s.c. Hilbert spaces). 
Moreover, $\mathcal B(\mathcal X,\mathcal Y)$, $\mathcal F(\mathcal X,\mathcal Y)$, $\mathcal K(\mathcal X,\mathcal Y)$ and 
$\mathcal B^1(\mathcal X,\mathcal Y)$ denote the set of all bounded, finite-rank, compact and trace-class
operators between $\mathcal X$ and $\mathcal Y$, respectively. 

Scalar products are conjugate linear in the first argument and linear in the second one. Finally,
for an arbitrary set $S$, the terms $\overline{S}$ and $\operatorname{conv}(S)$ stand for its
closure and convex hull, respectively.

\subsection{Infinite-dimensional Hilbert Spaces and the Trace Class}
For a comprehensive introduction to infinite-dimensional Hilbert spaces and trace-class operators, we refer to, e.g.~\cite{berberian1976} and \cite{MeiseVogt}, respectively. Here, we recall only some
basic results which we will use frequently throughout this paper. \medskip

Let $(e_i)_{i\in I}$ be any orthonormal basis of $\mathcal X$ and let $x \in \mathcal X$. Then one 
has the well-known \textit{Fourier expansion} 
\begin{equation*}
x=\sum_{i\in I} \langle e_i,x\rangle e_i\,,
\end{equation*}
as well as \textit{Parseval's identity} 
\begin{equation*}
\sum_{i\in I}|\langle e_i,x\rangle|^2 = \Vert x\Vert^2
\end{equation*}
which reduces to \textit{Bessel's inequality} 
\begin{equation*}
\sum_{j\in J}|\langle f_j,x\rangle|^2  \leq  \Vert x\Vert^2
\end{equation*}
if $(f_j)_{j\in J}$ is any orthonormal system in $\mathcal X$ instead of an orthonormal basis.
Moreover, one has the following characterization and properties of unitary operators acting on $\mathcal X$:
\begin{itemize}
\item
$U\in\mathcal B(\mathcal X)$ is unitary if and only if $(Ue_i)_{i \in I}$ is an orthonormal basis of
$\mathcal X$.\vspace{4pt}
\item
The image $(Uf_j)_{j \in J}$ under a unitary operator $U$ again is an orthonormal system.\vspace{4pt}
\item
For any two finite orthonormal systems $(f_j)_{j = 1, \dots, n}$ and $(g_j)_{j = 1, \dots, n}$ there
exists unitary operator $U$ with $Uf_j = g_j$ for $j = 1, \dots, n$.
\end{itemize}

Generalizing the trace concept from finite-dimensional to infinite-dimensional Hilbert spaces leads
to the notion of trace-class operators. We need only the following two key results as can be found in, e.g. \cite[Chapter 16]{MeiseVogt}.

\begin{lemma}[Schmidt decomposition]\label{thm_1}
For each $C \in \mathcal K(\mathcal X,\mathcal Y)$, there exists a decreasing null sequence 
$(s_n(C))_{n\in\mathbb N}$ in $[0,\infty)$ and orthonormal systems $(f_n)_{n\in\mathbb N}$ in $\mathcal X$
and $(g_n)_{n\in\mathbb N}$ in $\mathcal Y$ such that
\begin{align*}
C = \sum_{n=1}^\infty s_n(C)\langle f_n,\cdot\rangle g_n\,,
\end{align*}
where the series converges in the operator norm.
\end{lemma}

\noindent
Then the \emph{trace class} $\mathcal B^1(\mathcal X,\mathcal Y)$ is defined by
\begin{align*}
\mathcal B^1(\mathcal X,\mathcal Y)
:= \Big\lbrace C \in\mathcal K(\mathcal X,\mathcal Y)\,\Big|\,\sum\nolimits_{n=1}^\infty s_n(C)<\infty\Big\rbrace\,.
\end{align*}
The \emph{singular numbers} $(s_n(C))_{n\in\mathbb N}$ in Lemma \ref{thm_1} are uniquely determined by $C$. 
However, this is obviously not true for the orthonormal systems $(f_n)_n$ and $(g_n)_n$. Furthermore, the trace norm
\begin{align*}
\nu_1(C) := \sum_{n=1}^\infty s_n(C)
\end{align*}
turns $\mathcal B^1(\mathcal X,\mathcal Y)$ into a Banach space. The trace class $\mathcal B^1(\mathcal X)$
constitutes -- just like the compact operators -- a two-sided ideal in the $C^*$-algebra of all bounded operators
$\mathcal B(\mathcal X)$. The next result is a simple consequence of \cite[Lemma 16.6.(6)]{MeiseVogt}.

\begin{lemma}\label{lemma_10}
For any $S,T\in\mathcal B(\mathcal X)$ and any $C\in\mathcal B^1(\mathcal X)$, one has
\begin{align*}
\nu_1(SCT)\leq\Vert S\Vert\nu_1(C)\Vert T\Vert\,.
\end{align*}
\end{lemma}

\noindent
Now for arbitrary $C\in\mathcal B^1(\mathcal X)$, the trace of $C$ is defined via
\begin{align}\label{eq:trace}
\operatorname{tr}(C):=\sum\nolimits_{i\in I}\langle f_i,Cf_i\rangle\,,
\end{align}
where $(f_i)_{i\in I}$ can be any orthonormal basis of $\mathcal X$. The trace is well-defined as
one can show that the right-hand side of \eqref{eq:trace} does not depend on the choice of 
$(f_i)_{i\in I}$. Important properties are
\begin{align}
\operatorname{tr}(CT)&=\operatorname{tr}(TC)\label{eq:4_a}\\
\operatorname{tr}\big((\langle x,\cdot\rangle y)T\big)&=\langle x,Ty\rangle\nonumber\\
|\operatorname{tr}(CT)|&\leq\nu_1(C)\Vert T\Vert\label{eq:4}
\end{align}
for all $C\in\mathcal B^1(\mathcal X)$, $T\in\mathcal B(\mathcal X)$ and $x,y\in\mathcal X$.

\subsection{Set Convergence}
\label{subsec:set-convergence}

In order to transfer the known results about convexity and star-shapedness of the $C$-numerical
range of matrices to trace-class operators, we need some basic facts about set convergence. 
We will use the Hausdorff metric on compact subsets (of $\mathbb C$) and the associated notion
of convergence, see, e.g.~\cite{nadler1978}.\medskip

\noindent
The distance between $z \in \mathbb C$ and any non-empty compact subset $A \subseteq \mathbb C$ is
defined by
\begin{align}\label{eq.Hausdorff-1}
d(z,A) := \min_{w \in A} d(z,w) = \min_{w \in A} |z-w|\,.
\end{align}
Based on \eqref{eq.Hausdorff-1}, the \emph{Hausdorff metric} $\Delta$ on the set of all non-empty
compact subsets of $\mathbb C$ is given by
\begin{align*}
\Delta(A,B) := \max\Big\lbrace \max_{z \in A}d(z,B),\max_{z \in B}d(z,A) \Big\rbrace\,.
\end{align*}

\noindent
The following characterization of the Hausdorff metric will be essential throughout this paper.

\begin{lemma}\label{lemma_11}
Let $A,B \subset \mathbb C$ be two non-empty compact sets and let $\varepsilon > 0$. 
Then $\Delta(A,B) \leq \varepsilon$ if and only if for all $z \in A$, there exists $w \in B$ 
with $d(z,w) \leq \varepsilon$ and vice versa.
\end{lemma}
\begin{proof}
By definition, $\Delta(A,B) \leq \varepsilon$ is equivalent to $\max_{z \in A}d(z,B)\leq\varepsilon$ and $\max_{z \in B}d(z,A) \leq\varepsilon$. This in turn means
\begin{align}\label{eq:hausdorff}
\max_{z \in A}\min_{w \in B} d(z,w)\leq\varepsilon\qquad\text{and}\qquad\max_{z \in B}\min_{w \in A} d(z,w) \leq\varepsilon\,.
\end{align}
Evidently, \eqref{eq:hausdorff} holds if and only if for all $z \in A$, there exists $w \in B$ 
with $d(z,w) \leq \varepsilon$ and vice versa.
\end{proof}

With this metric at hand, one can introduce the notion of convergence of a sequence 
$(A_n)_{n\in\mathbb N}$ of non-empty compact subsets. Alternatively, one can introduce the notion of Kuratowski convergence as
follows:\medskip

\noindent
Consider a sequence $(A_n)_{n\in\mathbb N}$ of non-empty compact subsets of $\mathbb C$ and define
\begin{itemize}
\item 
$\liminf_{n\to\infty} A_n$ as the set of all $z \in \mathbb C$ such that for all $\varepsilon > 0$ 
one has $B_\varepsilon(z)\cap A_n\neq\emptyset$ for all but finitely many indices.\vspace{4pt}
\item
$\limsup_{n\to\infty} A_n$ as the set of all $z \in \mathbb C$ such that for all $\varepsilon>0$
one has $B_\varepsilon(x)\cap A_n\neq\emptyset$ for infinitely many indices.
\end{itemize}

\noindent
If $\liminf_{n\to\infty} A_n = \limsup_{n\to\infty} A_n =:A$ one says that $(A_n)_{n\in\mathbb N}$ converges 
to $A$ and writes
\begin{align*}
\lim_{n\to\infty}A_n=A\,.
\end{align*}

\noindent
The following Lemma shows that both approaches are essentially equivalent, cf.~\cite[Thm. 0.7]{nadler1978}.

\begin{lemma}\label{lem:Hausdorff}
Let $(A_n)_{n\in\mathbb N}$ be a bounded sequence of non-empty compact subsets of $\mathbb C$.
\begin{itemize}
\item[(a)]
If $(A_n)_{n\in\mathbb N}$ converges to $A$ with respect to the Hausdorff metric, then
$\liminf_{n\to\infty} A_n = \limsup_{n\to\infty} A_n = A$.\vspace{4pt}
\item[(b)]
If $\liminf_{n\to\infty} A_n = \limsup_{n\to\infty} A_n =:A$, then $A$ is non-empty and compact and 
$(A_n)_{n\in\mathbb N}$ converges to $A$ with respect to the Hausdorff metric.
\end{itemize}
\end{lemma}

\noindent
For reference, we finally state the following result which will be used frequently below.
A proof can be found in Appendix \ref{appendix_A}.

\begin{lemma}\label{lemma_5}
Let $(A_n)_{n\in\mathbb N}$ and $(B_n)_{n\in\mathbb N}$ be bounded sequences of non-empty compact subsets
of $\mathbb C$ such that $\lim_{n\to\infty}A_n = A$, $\lim_{n\to\infty}B_n = B$ and let $(z_n)_{n\in\mathbb N}$ be any sequence of complex numbers with $\lim_{n\to\infty}z_n = z$. Then the following statements hold.
\begin{itemize}
\item[(a)] 
If $A_n\subseteq B_n$ for all $n\in\mathbb N$, then $A \subseteq B$.\vspace{4pt}
\item[(b)] 
The sequence $(\operatorname{conv}(A_n))_{n\in\mathbb N}$ of compact subsets converges to $\operatorname{conv}(A)$, i.e.
\begin{align*}
\lim_{n\to\infty}\operatorname{conv}(A_n) = \operatorname{conv}(A)\,.
\end{align*}
\item[(c)] 
If $A_n$ is convex for all $n\in\mathbb N$, then $A$ is convex.\vspace{4pt}
\item[(d)] 
If $A_n$ is star-shaped with respect to $z_n$ for all $n\in\mathbb N$, then $A$ is star-shaped with respect to $z$.
\end{itemize}
\end{lemma}

\section{Results}\label{sec:results}

Let $\mathcal H$ denote an arbitrary infinite-dimensional separable complex (i.s.c.) Hilbert space. We define
the $C$-numerical range $W_C (T)$ of a bounded linear operator $T$ on $\mathcal H$, where $C$ can
be any trace-class operator on $\mathcal H$, as follows.

\begin{definition}\label{defi_1}
For any $C\in\mathcal B^1(\mathcal H),T\in\mathcal B(\mathcal H)$, let
\begin{align*}
W_C (T):=\lbrace \operatorname{tr}(CU^\dagger TU)\,|\,U\in\mathcal B(\mathcal H)\text{ unitary}\rbrace\,.
\end{align*}
\end{definition}

Throughout this paper we need some formalism to associate matrices with bounded operators on $\mathcal H$ 
and vice versa. In doing so, let $(e_n)_{n\in\mathbb N} $ be some orthonormal basis of $\mathcal H$ and let
$(\hat e_i)_{i=1}^n$ be the standard basis of $\mathbb C^n$. For any $n\in\mathbb N$ we define 
\begin{align}\label{Gamma}
\Gamma_n:\mathbb C^n\to \mathcal H,\qquad \hat{e_i}\mapsto \Gamma_n(\hat e_i):=e_i
\end{align}
and its linear extension to all of $\mathbb C^n$. Now let 
\begin{align*}
E_n:\mathbb C^{n\times n}\to\mathcal B(\mathcal H),\qquad A\mapsto E_n(A):=\Gamma_n A\Gamma_n^\dagger
\end{align*}
be the embedding of $\mathbb C^{n\times n}$ into $B(\mathcal H)$ relative to the basis $(e_n)_{n\in\mathbb N} $
and let
\begin{align}\label{cut_out_operator}
[\;\cdot\;]_n:\mathcal B(\mathcal H)\to\mathbb C^{n\times n},\qquad A\mapsto [A]_n:=\Gamma_n^\dagger A\Gamma_n
\end{align}
be the operator which ``cuts out'' the upper $n\times n$ block of (the matrix representation of) $A$ 
with respect to $(e_n)_{n\in\mathbb N} $. 

\begin{remark}
Obviously, $W_{E_n(C)}(T)$ coincides with \eqref{eq:jones_1} for all $C\in\mathbb C^{n\times n}$
and $T\in\mathcal B(\mathcal H)$, where $E_n$ is the embedding operator with respect to any orthonormal basis of $\mathcal H$. Thus Definition \ref{defi_1} or, equivalently, Eq. \eqref{eq:WCA-inf-dim} actually generalize Jones' approach \cite{article_jones} who, in our words, considered only finite-rank operators $C\in\mathcal F(\mathcal H)$.
\end{remark}

The following lemma which will be needed later is a trivial consequence of the standard trace 
identity (\ref{eq:4_a}) for operators acting on the \emph{same} Hilbert space.

\begin{lemma}\label{embedding_trace_preserv}
Let $n\in\mathbb N$, $A\in\mathbb C^{n\times n}$, $B\in\mathcal B(\mathcal H)$ and any orthonormal bases $(e_n)_{n\in\mathbb N}$, $(g_n)_{n\in\mathbb N}$ of $\mathcal H$ be given. Then
\begin{align*}
\operatorname{tr}\big((\Gamma^g_n)^\dagger B\Gamma_n^eA\big)=\operatorname{tr}\big(B\Gamma_n^eA(\Gamma_n^g)^\dagger\big)
\end{align*}
where $\Gamma_n^e$ ($\Gamma_n^g$) is the above embedding $\Gamma_n$ with respect to $(e_n)_{n\in\mathbb N}$ ($(g_n)_{n\in\mathbb N}$).
\end{lemma}

\begin{proof}
Consider the operators 
\begin{align*}
\begin{pmatrix}
B & 0 \\ 0 & 0
\end{pmatrix},
\qquad
\begin{pmatrix}
0 & \Gamma_n^e \\ (\Gamma_n^g)^\dagger & 0
\end{pmatrix},
\quad\text{and}\quad
\begin{pmatrix}
0 & 0 \\ 0 & A
\end{pmatrix}
\end{align*}
acting on $\mathcal H \times \mathbb C^n$ and use the standard cyclicity result of the trace.
\end{proof}

\subsection{Convexity and Star-shapedness}

Our strategy is to transfer the well-known properties of the finite-dimensional $[C]_n$-numerical range of 
$[T]_n$ to $W_C(T)$ via the convergence results of Lemma \ref{lemma_5}.\medskip

Let $B\in\mathcal B(\mathcal H)$ and let $(e_n)_{n\in\mathbb N}$ be an orthonormal basis of $\mathcal H$. 
For any $k\in\mathbb N$ we define the $k$-th block approximation of $B$ with respect to $(e_n)_{n\in\mathbb N}$ as 
\begin{align}\label{eq:defi_block}
B_k:=\Pi_kB\Pi_k,
\quad\text{where}\quad \Pi_k := \sum_{j=1}^k\langle e_j,\cdot\rangle e_j
\end{align}
is the orthogonal projection onto $\operatorname{span}\lbrace e_1,\ldots, e_k\rbrace$. Thus one has
\begin{align*}
B_k = \sum_{i,j=1}^k\langle  e_i, Be_j\rangle\langle e_j,\cdot\rangle e_i\,.
\end{align*}

\begin{lemma}\label{lemma_proj_strong_conv}
\begin{itemize}
\item[(a)] Let $(e_n)_{n\in\mathbb N}$ be any orthonormal basis of $\mathcal H$.
The sequence of orthogonal projections $(\Pi_n)_{n\in\mathbb N}$ given by \eqref{eq:defi_block}
converges strongly to the identity operator $\operatorname{id}_{\mathcal H}$ on $\mathcal H$.\vspace{4pt}
\item[(b)] Let $C \in \mathcal B^1(\mathcal H)$ and let $(S_n)_{n\in\mathbb N}$ be a sequence in 
$\mathcal B(\mathcal H)$ which converges strongly to $S \in\mathcal B(\mathcal H)$. Then one has
$S_n C \to SC$, $CS_n^\dagger \to CS^\dagger$, and $S_nCS_n^\dagger \to SCS^\dagger$ for $n \to \infty$
with respect to the trace-norm $\nu_1$. \vspace{4pt} 
\item[(c)] Let $T\in\mathcal K(\mathcal H)$ and let $(S_n)_{n\in\mathbb N}$ be a sequence in 
$\mathcal B(\mathcal H)$ which converges strongly to $S \in\mathcal B(\mathcal H)$. 
Then one has $S_nT \to ST$, $TS_n^\dagger \to TS^\dagger$ and $S_nTS_n^\dagger \to STS^\dagger$ 
for $n \to \infty$ with respect to the operator norm $\Vert \cdot \Vert$. 
\end{itemize} 
\end{lemma}

\begin{proof}
(a) This follows from the Fourier expansion and Parseval's identity.\medskip

\noindent (b) The case $C=0$ is obvious. Therefore, we can assume w.l.o.g.~$C\neq 0$.
By the uniform boundedness
principle, the sequence $(\Vert S_n\Vert)_{n\in\mathbb N}$ is bounded and thus there exists $\kappa > 0$
such that $\Vert S\Vert\leq \kappa$ and $\Vert S_n\Vert\leq \kappa$ for all $n\in\mathbb N$. Now let
$\varepsilon>0$ be given. By Lemma \ref{thm_1}, there exist orthonormal systems $(e_n)_{n\in\mathbb N}$
and $(f_n)_{n\in\mathbb N}$ in $\mathcal H$ with $C=\sum_{k=1}^\infty s_k(C)\langle e_k,\cdot\rangle f_k$ and
$\nu_1(C)=\sum_{k=1}^\infty s_k(C) < \infty$. Hence we can choose $K\in\mathbb N$ such that
\begin{align*}
\sum_{k=K+1}^\infty s_k(C)<\frac{\varepsilon}{3\kappa}\,.
\end{align*}
Based on this, we decompose $C = C_1 + C_2$ via 
\begin{align*}
C_1 := \sum_{k=1}^K s_k(C)\langle e_k,\cdot\rangle f_k
\quad\text{and}\quad
C_2 := \sum_{k=K+1}^\infty s_k(C)\langle e_k,\cdot\rangle f_k\,.
\end{align*}
Note that $C_1$ is finite-rank hence trace class, even if $C$ was only compact. Now together with Lemma \ref{lemma_10} we get
\begin{align}\label{eq_decomposition}
\nu_1(SC &- S_nC)=\nu_1(SC_1+SC_2-S_n C_1 -S_n C_2 )\\
&\leq \nu_1(SC_1-S_n C_1)+\Vert S\Vert\nu_1(C_2)+\Vert S_n\Vert\nu_1(C_2)
<\nu_1(SC_1-S_nC_1 )+\frac{2\varepsilon}{3}\,.\nonumber
\end{align}
Now our goal is to choose $N\in\mathbb N$ such that $\nu_1(SC_1-S_nC_1)$ is smaller than $\varepsilon/3$ 
for all $n\geq N$. Note that $\nu_1(\langle x,\cdot\rangle y)=\Vert x\Vert\Vert y\Vert$ for any 
$x,y\in\mathcal H$. Hence it follows
\begin{align*}
\nu_1(SC_1-S_n C_1)
&=\nu_1\Big( \sum_{k=1}^K s_k(C)\langle e_k,\cdot\rangle Sf_k-\sum_{k=1}^K s_k(C)\langle e_k,\cdot\rangle S_nf_k \Big)\\
&\leq  \sum_{k=1}^K s_k(C)\nu_1\big(\langle e_k,\cdot\rangle (Sf_k-S_nf_k)\big)=\sum_{k=1}^K s_k(C) \Vert Sf_k-S_nf_k \Vert\,.
\end{align*}
Moreover, the strong convergence of $(S_n)_{n\in\mathbb N}$ yields $N \in \mathbb N$ such that
\begin{align*}
\Vert Sf_k - S_nf_k \Vert<\frac{\varepsilon}{3\nu_1(C_1)}
\end{align*}
for $k = 1, \dots, K$ and all $n\geq N$. Finally, for $n \geq N$ we get
\begin{align*}
\nu_1(SC_1-S_n C_1) < \frac{\varepsilon}{3\nu_1(C_1)}\sum_{k=1}^K s_k(C)  = \frac{\varepsilon}{3}
\end{align*}
which implies $\nu_1(SC-S_nC) \to 0$ for $ n\to\infty$. The case $(\nu_1(CS^\dagger-CS_n^\dagger))_{n\in\mathbb N}$ 
follows immediately from the identity $\nu_1(A) = \nu_1(A^\dagger)$ for all $A \in \mathcal B^1(\mathcal H)$.
Combining both results and Lemma \ref{lemma_10} yields
\begin{align*}
\nu_1(SCS^\dagger &- S_nCS_n^\dagger) \leq \nu_1\big(S(CS^\dagger- CS_n^\dagger)\big) + \nu_1\big((SC-S_nC)S_n^\dagger\big)\\
& \leq \Vert S\Vert\nu_1(CS^\dagger-CS_n^\dagger)+\nu_1(SC-S_nC) \kappa \to 0 
\quad\text{for}\quad n\to\infty\,.
\end{align*}

\noindent (c) Finally, let $T\in\mathcal K(\mathcal H)$. Again Lemma \ref{thm_1} guarantees a Schmidt decomposition 
$T=\sum_{k=1}^\infty s_k(T)\langle g_k,\cdot\rangle h_k$. A straightforward application of Bessel's inequality
combined with the monotonicity of the singular numbers $s_k(T)$ implies
\begin{align*}
\Big\Vert\sum_{k=m}^\infty s_k(T)\langle g_k,\cdot\rangle h_k\Big\Vert^2 \leq s_m^2(T) \to 0\quad\text{for}\quad m \to \infty\,.
\end{align*}
Based on this observation, one can proceed as in part (b). More precisely, a decomposition as in  
\eqref{eq_decomposition} and the idenity $\Vert \langle x,\cdot\rangle y\Vert = \Vert x\Vert\Vert y\Vert$ 
for all $x,y\in\mathcal H$ will yield the desired result.
\end{proof}

\begin{lemma}\label{strong_tr_conv}
Let $(S_n)_{n\in\mathbb N}$ be a sequence in $\mathcal B(\mathcal H)$ which converges strongly to 
$S\in\mathcal B(\mathcal H)$. Then for all $C\in\mathcal B^1(\mathcal H)$ and $T\in\mathcal B(\mathcal H)$
one has
\begin{align*}
\lim_{n\to\infty}\operatorname{tr}(CS_n^\dagger TS_n)=\operatorname{tr}(CS^\dagger TS)\,.
\end{align*}
Furthermore, 
\begin{itemize}
\item the sequence of linear functionals $\big(\operatorname{tr}(CS_n^\dagger(\cdot)S_n)\big)_{n\in\mathbb N}$ 
converge uniformly to $\operatorname{tr}(CS^\dagger (\cdot)S)$ on bounded subsets of $\mathcal B(\mathcal H)$.
\item the sequence of linear functionals $\big(\operatorname{tr}((\cdot)S_n^\dagger TS_n)\big)_{n\in\mathbb N}$
converge uniformly to $\operatorname{tr}((\cdot)S^\dagger TS)$ on compact subsets of $\mathcal B^1(\mathcal H)$.
\end{itemize}
If $T$ additionally is compact, then $\big(\operatorname{tr}((\cdot)S_n^\dagger TS_n)\big)_{n\in\mathbb N}$
converges uniformly to $\operatorname{tr}((\cdot)S^\dagger TS)$ on (trace norm-) bounded subsets of $\mathcal B^1(\mathcal H)$.
\end{lemma}

\begin{proof}
This is a simple consequence of (\ref{eq:4}) and Lemma \ref{lemma_proj_strong_conv} (b) as
\begin{align*}
|\operatorname{tr}(CS^\dagger TS) &-\operatorname{tr}(CS_n^\dagger TS_n)| = 
\big|\operatorname{tr}\big((SCS^\dagger-S_nCS_n^\dagger)T\big)\big|\\
&\leq\Vert T\Vert \nu_1(SCS^\dagger-S_nCS_n^\dagger) \to 0
\quad\text{for}\quad n\to\infty\,.
\end{align*}
The remaining assertions of the lemma are evident.
\end{proof}

\begin{remark}
Note that for arbitrary bounded operators $T$, $\operatorname{tr}((\cdot)S_n^\dagger TS_n)$ does not necessarily converge uniformly to
$\operatorname{tr}((\cdot)S^\dagger TS)$ on (trace norm-) bounded subsets of $\mathcal B^1(\mathcal H)$. 
A counter-example is given in Appendix \ref{appendix_examples} (Ex. \ref{ex_1}).
\end{remark}

\begin{lemma}\label{U-approximation}
Let $U \in\mathcal B(\mathcal H)$ be unitary and consider orthonormal bases $(e_n)_{n \in \mathbb N}$, $(g_n)_{n \in \mathbb N}$ of $\mathcal H$. Then there exists a
sequence $(\hat{U}_n)_{n \in \mathbb N}$ in $\mathcal B(\mathcal H)$ which satisfies the following properties:
\begin{itemize}
\item[(a)]
$(\hat{U}_n)_{n \in \mathbb N}$ converges strongly to $U$.\vspace{4pt}
\item[(b)]
$\Pi^g_{2n}\hat{U}_n\Pi^e_{2n}=\hat U_n$ for all $n \in \mathbb N$.\vspace{4pt}
\item[(c)]
$(\Gamma_{2n}^g)^\dagger \hat{U}_n \Gamma_{2n}^e\in\mathbb C^{2n\times 2n}$ is unitary for all $n \in \mathbb N$.
\end{itemize}
Here, $\Gamma_k^e$, $\Pi_{k}^e$ and $\Gamma_k^g$, $\Pi_k^g$ are the maps given by \eqref{Gamma} and \eqref{eq:defi_block} with respect to $(e_n)_{n\in\mathbb N}$ and $(g_n)_{n\in\mathbb N}$, respectively.
\end{lemma}
\noindent As the proof of Lemma \ref{U-approximation} is rather technical we here refer to Appendix \ref{App_Lemma_U}.

\begin{lemma}\label{lemma_3}
Let $C\in\mathcal B^1(\mathcal H)$ and $T\in\mathcal B(\mathcal H)$ and let $(e_n)_{n\in\mathbb N}$, $(g_n)_{n\in\mathbb N}$ be
arbitrary orthonormal bases of $\mathcal H$. Furthermore, $[\,\cdot\,]_k^e$ and $[\,\cdot\,]_k^g$ are the maps given by \eqref{cut_out_operator} with respect to $(e_n)_{n\in\mathbb N}$ and $(g_n)_{n\in\mathbb N}$, respectively. Then for all $\varepsilon>0$ and $w\in \overline{W_C(T)}$, there exists $N\in\mathbb N$ such that the distance $d(w,W_{[C]^e_{n}}([T]^g_{n}))<\varepsilon$ for all $n\geq N$.
\end{lemma}
\begin{proof}
Let $\varepsilon>0$ and let $w \in \overline{W_C(T)}$ be given. Then there exists unitary $U\in\mathcal B(\mathcal H)$ such that
$|w-\operatorname{tr}(CU^\dagger TU)|<\varepsilon/2$. By Lemma \ref{U-approximation}, we can find a sequence $(\hat{U}_n)_{n \in \mathbb N}$ which converges strongly to $U$. Lemma \ref{strong_tr_conv}
then yields  $N\in\mathbb N$ such that
\begin{align*}
|\operatorname{tr}(CU^\dagger TU)-\operatorname{tr}(C\hat U_n^\dagger T\hat U_n)|<\frac{\varepsilon}{2}
\end{align*}
for all $n\geq N$. Using Lemma \ref{embedding_trace_preserv} and \ref{U-approximation}, one gets
\begin{align*}
\operatorname{tr}(C\hat U_n^\dagger T\hat U_n)&= \operatorname{tr}\big( C(\Pi^g_{2n}\hat U_n\Pi^e_{2n})^\dagger T(\Pi^g_{2n}\hat U_n\Pi^e_{2n})\big)\\
&=\operatorname{tr}\big([C]^e_{2n}((\Gamma_{2n}^g)^\dagger \hat{U}_n \Gamma_{2n}^e)^\dagger [T]^g_{2n}(\Gamma_{2n}^g)^\dagger \hat{U}_n \Gamma_{2n}^e\big)\in W_{[C]^e_{n}}([T]^g_{n})\,.
\end{align*}
Thus $|w-\operatorname{tr}(C\hat U_n^\dagger T\hat U_n)|<\varepsilon$ for all $n\geq N$, which concludes the proof as the $C$-numerical 
range of any pair of matrices is compact \cite[(2.5)]{article_li_radii}. 
\end{proof}

\noindent Note that in the above proof, $N$ depends usually on $\varepsilon$ but also on the chosen point
$w \in \overline{W_C(T)}$.

\begin{theorem}\label{lemma_2}
Let $C\in\mathcal B^1(\mathcal H)$ and $T\in\mathcal B(\mathcal H)$ and let $(e_n)_{n\in\mathbb N}$, $(g_n)_{n\in\mathbb N}$ be
arbitrary orthonormal bases of $\mathcal H$. Furthermore, $[\,\cdot\,]_k^e$, $\Pi_{k}^e$ and $[\,\cdot\,]_k^g$, $\Pi_k^g$ for all $k\in\mathbb N$ are the maps \eqref{cut_out_operator} and \eqref{eq:defi_block} with respect to $(e_n)_{n\in\mathbb N}$ and $(g_n)_{n\in\mathbb N}$, respectively. Then
\begin{align*}
\lim_{n\to\infty}W_{[C]^e_{2n}}([T]^g_{2n})=\overline{W_C(T)}=\lim_{n\to\infty}\overline{W_{\Pi_n^e C\Pi_n^e}(T)}\,,
\end{align*}
where $W_{[C]^e_{2n}}([T]^g_{2n})$ denotes the ordinary $[C]^e_{2n}$-numerical range of $[T]^g_{2n}$ as defined in (\ref{c_num_range_findim}). If $T$ is additionally compact, then
\begin{align}\label{eq:lemma_2_2}
\lim_{n\to\infty}\overline{W_{\Pi_n^e C\Pi_n^e}(\Pi_n^g T\Pi_n^g)}=\overline{W_C(T)}\,.
\end{align}
\end{theorem}

\begin{proof}
As we want to check convergence with respect to the Hausdorff metric, we have to make sure that all occuring sets are non-empty and compact. The non-empty sets $\overline{W_{\Pi_n^e C\Pi_n^e}(T)}$, $\overline{W_{\Pi_n^e C\Pi_n^e}(\Pi_n^g T\Pi_n^g)}$ are bounded by $\nu_1(C) \Vert T\Vert$ due to
$|\operatorname{tr}(\Pi_n^e C\Pi_n^eU^\dagger \Pi_n^g T\Pi_n^gU)| \leq \nu_1(\Pi_n^e C\Pi_n^e) \Vert \Pi_n^g T\Pi_n^g\Vert \leq \nu_1(C) \Vert T\Vert$ and thus all of them are compact. Here we used
$\Vert U \Vert = \Vert \Pi_n^e\Vert =\Vert \Pi_n^g\Vert = 1$. Again, the $C$-numerical 
range of any pair of matrices is also compact \cite[(2.5)]{article_li_radii}.\medskip

\noindent
The case $C=0$ or $T=0$ is obvious, hence w.l.o.g.~we can assume $C,T\neq 0$. First, we prove the equality
\begin{align*}
\lim_{n\to\infty}W_{[C]^e_{2n}}([T]^g_{2n})=\overline{W_C(T)} \,.
\end{align*}
In view of Lemma \ref{lemma_11}, we have to consider two cases:\medskip

Let $\varepsilon>0$. Then due to compactness, there exist finitely many $w_1,\ldots,w_L\in \overline{W_C(T)}$ such that
\begin{align*}
\bigcup_{k=1}^L B_{\varepsilon/2}(w_k)\supset \overline{W_C(T)}
\end{align*}
where $B_{\varepsilon/2}(w_k)$ denotes open $\varepsilon/2$-balls around $w_k$. By Lemma \ref{lemma_3}, each of these $w_k$ admits $N_k\in\mathbb N$ such that $d(w_k,W_{[C]^e_{n}}([T]^g_{n}))<\varepsilon/2$ for all $n\geq N_k$. Define $N':=\max\lbrace N_1,\ldots,N_L\rbrace$. Now for any $w\in \overline{W_C(T)}$, there exists $k\in\lbrace 1,\ldots,L\rbrace$ such that $|w-w_k|<\varepsilon/2$ and thus
\begin{align*}
d(w,W_{[C]^e_{n}}([T]^g_{n}))\leq |w-w_k|+d(w_k,W_{[C]^e_{n}}([T]^g_{n}))<\varepsilon
\end{align*}
for all $n\geq N'$.

On the other hand, for $G_{2n} := \sum_{k=2n+1}^\infty \langle e_k,\cdot\rangle g_k$ it is easy to see that $(G_{2n})_{n\in\mathbb N}$ converges strongly to the zero operator. By Lemma \ref{lemma_proj_strong_conv} (b) we obtain $N''\in\mathbb N$
such that
\begin{align*}
\max\lbrace\nu_1(CG_{2n}),\nu_1(G_{2n}C))\rbrace<\frac{\varepsilon}{3\|T\|}
\end{align*}
for all $n\geq N''$. Now let $v_n\in W_{[C]^e_{2n}}([T]^g_{2n})$, i.e.~there exists unitary $U_n\in\mathbb C^{2n\times 2n}$
such that $v_n = \operatorname{tr}([C]_{2n}^e U_n^\dagger [T]_{2n}^g U_n  )$. Again, by Lemma \ref{embedding_trace_preserv}, we get 
$v_n = \operatorname{tr}\big(C( \Gamma_{2n}^gU_n(\Gamma_{2n}^e)^\dagger)^\dagger T \Gamma_{2n}^gU_n(\Gamma_{2n}^e)^\dagger\big)$. Next, we define the operator
\begin{align*}
\tilde U_n := \Gamma_{2n}^gU_n(\Gamma_{2n}^e)^\dagger + G_{2n} \in\mathcal B(\mathcal H)
\end{align*}
with $G_{2n} $ given as above. It is readily verified that $\tilde U_n$ is unitary and, therefore, we conclude
$\tilde{v}_n := \operatorname{tr}(C\tilde U^\dagger_n T\tilde U_n)\in W_C(T)$.
Via Lemma \ref{lemma_10} we finally obtain
\begin{align*}
|v_n-\tilde{v}_n|
&=|\operatorname{tr}\big(CG_{2n}T\Gamma_{2n}^gU_n(\Gamma_{2n}^e)^\dagger\big)+\operatorname{tr}\big(C(\Gamma_{2n}^gU_n(\Gamma_{2n}^e)^\dagger)^\dagger TG_{2n}\big)
+\operatorname{tr}(CG_{2n}TG_{2n})|\\
&\leq \big(\nu_1(CG_{2n})+\nu_1(G_{2n}C)+\nu_1(CG_{2n})\big)\Vert T\Vert<\varepsilon
\end{align*}
which yields $d(v_n,\overline{W_C(T)})<\varepsilon$ for all $n\geq N''$. Thus, choosing $N:=\max\lbrace N',N''\rbrace$, Lemma \ref{lemma_11} implies that the Hausdorff distance $\Delta(W_{[C]^e_n}([T]^g_n),\overline{W_C(T)})<\varepsilon$ for all $n\geq N$.

\medskip
Next, we tackle the equality
\begin{align*}
 \lim_{n\to\infty}\overline{W_{\Pi_n^e C\Pi_n^e}(T)}=\overline{W_C(T)}\,.
\end{align*}
Let $\varepsilon>0$ be given. By Lemma \ref{lemma_proj_strong_conv} there exists $\hat N\in\mathbb N$ such that
\begin{align*}
\nu_1(C-\Pi_n^e C\Pi_n^e)<\frac{\varepsilon}{2\Vert T\Vert}
\end{align*}
for all $n\geq\hat N$. For $w\in \overline{W_C(T)}$, there again exists unitary
$U\in\mathcal B(\mathcal H)$ such that $w' := \operatorname{tr}(CU^\dagger TU) \in W_C(T)$ satisfies 
$|w-w'|<\varepsilon/2$. Thus, for $w_n:=\operatorname{tr}(\Pi_n^e C\Pi_n^eU^\dagger TU)\in W_{\Pi_n^e C\Pi_n^e}(T)$ one has
\begin{align*}
|w-w_n|\leq |w-w'|+|w'-w_n|<
\frac{\varepsilon}{2}+\nu_1(C-\Pi_n^e C\Pi_n^e)\Vert U^\dagger TU\Vert < \varepsilon
\end{align*}
for all $n\geq N$.

On the other hand, let $v_n\in \overline{W_{\Pi_n^e C\Pi_n^e}(T)}$, i.e.~there exists unitary 
$U_n \in \mathcal B(\mathcal H)$ such that $v'_n := \operatorname{tr}(\Pi_n^e C\Pi_n^e U_n^\dagger T U_n)$
satisfies $|v_n - v_n'| < \varepsilon/2$. Moreover, for $\tilde{v}_n := \operatorname{tr}(C U_n^\dagger T U_n) \in W_{C}(T)$,
we obatin
\begin{align*}
|v_n-\tilde{v}_n| \leq |v_n-v'_n| + |v'_n-\tilde{v}_n|
< \frac{\varepsilon}{2} + \nu_1(C-\Pi_n^e C\Pi_n^e)\Vert U_n^\dagger T U_n\Vert < \varepsilon
\end{align*}
for all $n\geq N$. Again, Lemma \ref{lemma_11} implies $\lim_{n\to\infty}\overline{W_{\Pi_n^e C\Pi_n^e}(T)} = \overline{W_C(T)}$.\medskip 

\noindent Finally, let $T$ be additionally compact and $\varepsilon>0$. By Lemma \ref{lemma_proj_strong_conv} there exists $\tilde N\in\mathbb N$ such that
\begin{align*}
\|T-\Pi_n^g T\Pi_n^g\|<\frac{\varepsilon}{2\nu_1(C)}
\end{align*}
for all $n\geq\hat N$. As
\begin{align*}
|\operatorname{tr}(CU^\dagger TU) &- \operatorname{tr}(\Pi_n^e C\Pi_n^eU^\dagger \Pi_n^g T\Pi_n^gU)| \leq |\operatorname{tr}(CU^\dagger TU)-\operatorname{tr}(\Pi_n^e C\Pi_n^eU^\dagger TU)|\\
&+ |\operatorname{tr}(\Pi_n^e C\Pi_n^eU^\dagger TU)-\operatorname{tr}(\Pi_n^e C\Pi_n^eU^\dagger \Pi_n^g T\Pi_n^gU)|\,,
\end{align*}
one can choose $N:=\max\lbrace\hat N,\tilde N\rbrace$ to obtain as above $\Delta(\overline{W_{\Pi_n^e C\Pi_n^e}(\Pi_n^g T\Pi_n^g)},\overline{W_C(T)})<\varepsilon$ for all $n\geq N$. 
\end{proof}

\begin{remark}
In general, (\ref{eq:lemma_2_2}) does not hold for arbitrary bounded operators $T$ since -- even if the limit exists --  one has only the inclusion 
$\overline{W_C(T)} \subseteq \lim_{n\to\infty}\overline{W_{C_n}(T_n)}$ as the above
proof shows. A simple example 
which demonstrates this failing is given by Example \ref{ex_2} in Appendix \ref{appendix_examples}. 
\end{remark}

Now we are prepared to state and prove our first main result of this section.

\begin{theorem}\label{theorem_1a}
Let $C\in\mathcal B^1(\mathcal H)$ and $T\in\mathcal B(\mathcal H)$ be given. If $C$ is normal
with collinear eigenvalues or if $T$ is essentially self-adjoint, then $\overline{W_C(T)}$ is convex.
\end{theorem}

Recall, that a set in the complex plane is said to be \emph{collinear} if all of its elements lie on a common line.
Moreover, as in the matrix case, e.g.~\cite{article_marcus}, an operator $T \in \mathcal B(\mathcal H)$ is called \emph{essentially self-adjoint} if there exist $\theta\in\mathbb R$ and $\xi\in\mathbb C$ such 
that $e^{-i\theta} (T-\xi\operatorname{id}_{\mathcal H})$ is self-adjoint.

\begin{proof}
First, assume that $C$ is normal with collinear eigenvalues so as $C$ is compact as it is trace class, \cite[Thm. VIII.§4.6]{berberian1976} states that
there exists an orthonormal basis $(e_n)_{n\in\mathbb N}$ of $\mathcal H$ such that
$C=\sum_{n=1}^\infty\gamma_n\langle e_n,\cdot\rangle e_n$. By assumption, the eigenvalues\footnote{Note 
that $(\gamma_n)_{n\in\mathbb N}$ is the \textit{modified} eigenvalue sequence of $C$ as described
at the beginning of Section \ref{sect_C_spectrum}.}
$\gamma_n$ are collinear and $\gamma_n\to 0$ for $n\to\infty$ since $C$ is compact. This implies the 
existence of $\theta\in\mathbb R$ such that $e^{i\theta}\gamma_n\in\mathbb R$ for all $n\in\mathbb N$
and thus $e^{i\theta}C$ is self-adjoint.
By Theorem \ref{lemma_2}
\begin{align*}
\overline{W_{C}(T)}=\overline{W_{e^{i\theta}C}(e^{-i\theta}T)}=\lim_{n\to\infty}W_{[e^{i\theta}C]_{2n}}([e^{-i\theta}T]_{2n})
\end{align*}
where $[\,\cdot\,]_{2n}$ for all $n\in\mathbb N$ are the maps \eqref{cut_out_operator} with respect to $(e_n)_{n\in\mathbb N}$. Evidently, $[B]_n^\dagger=[B^\dagger]_n$ for all $B\in\mathcal B(\mathcal H)$ 
and all $n\in\mathbb N$. Therefore, $[e^{i\theta}C]_{2n}$ is hermitian and thus 
$W_{[e^{i\theta}C]_{2n}}([e^{-i\theta}T]_{2n})$ is convex for all $n\in\mathbb N$, cf.~\cite{article_poon}. 
Hence, Lemma \ref{lemma_5} (c) yields the desired result. The case $T$ being essentially self-adjoint
can be handled completely along the same line as then
\begin{align*}
W_C(T)= e^{i\theta} W_C(H)+\xi\operatorname{tr}(C)
\end{align*}
where $H:=e^{-i\theta} (T-\xi\operatorname{id}_{\mathcal H})$ is self-adjoint by definition.
\end{proof}

\begin{remark}
Unlike in finite dimensions, where $W_C(T)$ can be further located via the $C$-spectrum of $T$, it 
is intricate to obtain a similar result for infinite dimensions because there does not exist a 
meaningful counterpart of the $C$-spectrum for arbitrary bounded operators. However, if $T$ is compact
one can in fact define the $C$-spectrum of $T$ and generalize well-known properties of the matrix case, 
see Section \ref{sect_C_spectrum}.
\end{remark}

Before proceeding with the star-shapedness of $\overline{W_C(T)}$, we briefly recall the
definition\footnote{Some authors prefer a different definition which, however, is equivalent
to the stated one, cf.~\cite[Thm.~34.9]{bonsallduncan}.} of 
the \emph{essential numerical range} $W_e(T)$ of an operator $T \in \mathcal B(\mathcal H)$,
which can be given as follows
\begin{align*}
W_e(T) :=  \Big\lbrace \lim_{n \to \infty} \langle f_n,Tf_n\rangle\,\Big|\, (f_n)_{n \in \mathbb N}\;\text{ is ONS in }\mathcal H\Big\rbrace \subset \mathbb C\,.
\end{align*}
It is well known that $W_e(T)$ is a non-empty, convex and compact subset of $\mathbb C$, 
cf.~\cite[Thm.~34.2]{bonsallduncan}.

\begin{proposition}\label{prop_w_e}
Let $T\in\mathcal B(\mathcal H)$ and $\mu\in\mathbb C$ be given. The following statements are equivalent.
\begin{itemize}
\item[(a)] $\mu$ belongs to the essential numerical range $W_e(T)$, i.e.~there exists
an orthonormal system $(f_n)_{n \in \mathbb N}$ in $\mathcal H$ such that $\lim_{n \to \infty} \langle f_n,Tf_n\rangle = \mu$.\vspace{4pt}
\item[(b)] There exists an orthonormal system $(f_n)_{n\in\mathbb N}$ in $\mathcal H$ such that
\begin{align}\label{eq:theorem_app_0_1_1}
\lim_{n\to\infty}\frac{1}{n}\sum_{j=1}^n\langle f_j,Tf_j\rangle=\mu\,.
\end{align}
\item[(c)] There exists an orthonormal basis $(e_n)_{n\in\mathbb N}$ of $\mathcal H$ such that
\begin{align}\label{eq:theorem_app_0_1_2}
\lim_{n\to\infty}\frac{1}{n}\sum_{j=1}^n\langle e_j,Te_j\rangle=\mu\,.
\end{align}
\end{itemize}
\end{proposition}

\begin{proof}
(a) $\Longrightarrow$ (b): It is well known that the limit of a convergent sequence and the limit of its Ces\`aro mean are equal.

\medskip
\noindent
(b) $\Longrightarrow$ (a): Consider any orthonormal system $(f_n)_{n\in\mathbb N}$ which satisfies 
(\ref{eq:theorem_app_0_1_1}). We will show the relation
\begin{align}\label{eq:HP}
\mu \in \overline{\operatorname{conv}\big\lbrace \operatorname{HP} \big((\langle f_n,Tf_n\rangle)_{n\in\mathbb N}\big)\big\rbrace} =: E\,,
\end{align}
where $\operatorname{HP}(\cdot)$ denotes the set of all accumulation points of the respective sequence.
Once \eqref{eq:HP} is guaranteed we can conclude $\mu \in W_e(T)$ because the convexity and compactness
of $W_e(T)$ readily implies $E \subseteq W_e(T)$. Let us assume $\mu\notin E$. Since $E$ is obviously convex and compact, there exists a $\mathbb C$-linear 
functional $\varphi:\mathbb C\to\mathbb C$ with
\begin{align*}
\operatorname{Re}(\varphi(\mu)) < \min_{\lambda \in E}\operatorname{Re}(\varphi(\lambda))\,.
\end{align*}
Taking into account that the sequence $(\langle f_n,Tf_n\rangle)_{n\in\mathbb N}$ is bounded as $T$ is bounded,
a straightforward application of the Bolzano-Weierstra{\ss} Theorem shows that there exist only finitely 
many indices $n_1 < n_2  < \ldots < n_k \in \mathbb N$ such that 
\begin{align*}
\operatorname{Re}\big(\varphi(\langle f_{n_j},Tf_{n_j}\rangle)\big) \leq 
\frac{1}{2}\Big(\min_{\lambda \in E}\operatorname{Re}(\varphi(\lambda))+\operatorname{Re}(\varphi(\mu))\Big) =: \kappa
\end{align*}
for all $j \in\lbrace 1,\ldots,k\rbrace$. This yields the following contradicting estimate:
\begin{align*}
\operatorname{Re}(\varphi(\mu))&=\operatorname{Re}\Big(\varphi\Big( \lim_{n\to\infty}\frac{1}{n}\sum_{j=1}^n\langle f_j,Tf_j\rangle \Big)\Big)\\
& = \operatorname{Re}\Big(\varphi\Big( \lim_{n\to\infty}\frac{1}{n}\sum_{j=1}^{n_k}\langle f_j,Tf_j\rangle \Big)\Big) + \operatorname{Re}\Big(\varphi\Big( \lim_{n\to\infty}\frac{1}{n}\sum_{j=n_k+1}^n\langle f_j,Tf_j\rangle \Big)\Big)\\
&=\lim_{n\to\infty}\frac{1}{n}\sum_{j=n_k+1}^n\operatorname{Re}\big(\varphi(\langle f_j,Tf_j\rangle)\big)\geq 
\lim_{n\to\infty}\frac{\kappa (n - n_k)}{n} > \operatorname{Re}(\varphi(\mu))
\end{align*}
Hence, it follows $\mu\in E$.

\medskip
\noindent
(c) $\Longrightarrow$ (b): $\checkmark$

\medskip
\noindent
(b) $\Longrightarrow$ (c): Let $(f_n)_{n\in\mathbb N}$ be an orthonormal system in $\mathcal H$ such that
\eqref{eq:theorem_app_0_1_1} holds which we then extend to an orthonormal basis of $\mathcal H$. If,
in this procedure, we have to add only finitely many vectors (or none) we are obviously done. 
Therefore, we assume in the remaining part of the proof that we have to add countably infinitely
many vectors $(g_n)_{n\in\mathbb N}$. This allows us to define a new orthonormal basis $(e_n)_{n\in\mathbb N}$
by sorting $(g_n)_{n\in\mathbb N}$ into $(f_n)_{n\in\mathbb N}$ as follows: 
For $n = 2^k$ with $k\in\mathbb N$ choose $e_{n}=g_k$, while the gaps in between are filled up with 
the vectors of $(f_n)_{n\in\mathbb N}$, i.e.
\begin{align*}
(e_n)_{n\in\mathbb N}=(f_1,g_1,f_2,g_2,f_3,f_4,f_5,g_3,f_6, \ldots)\,.
\end{align*}
In doing so, for $2^k \leq n < 2^{k+1}$ we obtain the following identity
\begin{align*}
\frac1n\sum_{j=1}^n\langle e_j,Te_j\rangle=\Big(1-\frac{k}{n}\Big)\bigg(\frac{1}{n-k}\sum_{j=1}^{n-k}\langle f_j,Tf_j\rangle\bigg)+ \frac{1}{n}\sum_{j=1}^k\langle g_j,Tg_j\rangle\,.
\end{align*}
Obviously, $ \frac{k}{n}\to 0$ as $k \to \infty$ so 
\begin{align*}
\lim_{k\to\infty}\Big|\frac{1}{n}\sum_{j=1}^k\langle g_j,Tg_j\rangle\Big|
\leq\lim_{k\to\infty} \frac{k}{n}\Vert T\Vert=0
\end{align*}
and we conclude
\begin{align*}
\lim_{n \to \infty}\frac1n\sum_{j=1}^n\langle e_j,Te_j\rangle = 
\lim_{n \to \infty}\frac{1}{n-k}\sum_{j=1}^{n-k}\langle f_j,Tf_j\rangle = \mu
\end{align*}
as this is just a subsequence of \eqref{eq:theorem_app_0_1_1}.
\end{proof}

After these preliminaries, our second main result of this section reads as follows.

\begin{theorem}\label{theorem_1}
Let $C\in\mathcal B^1(\mathcal H)$ and $T\in\mathcal B(\mathcal H)$ be given. Then $\overline{W_C(T)}$
is star-shaped with respect to $\operatorname{tr}(C)W_e(T)$, i.e. all $z\in\operatorname{tr}(C)W_e(T)$ are star-centers of $\overline{W_C(T)}$.
\end{theorem}

\begin{proof}
Let any $\mu\in W_e(T)$. By Proposition \ref{prop_w_e} there exists an orthonormal basis $(e_n)_{n\in\mathbb N}$ of $\mathcal H$ such that (\ref{eq:theorem_app_0_1_2}) holds. Moreover, note that
\begin{align*}
\langle\hat e_j,[T]_{2n}\hat e_j\rangle=\langle\Gamma_{2n}\hat e_j,T\,\Gamma_{2n}\hat e_j\rangle=\langle e_j,Te_j\rangle
\end{align*}
for all $n\in\mathbb N$ and $j \in \lbrace 1,\ldots,2n\rbrace$, where $[\,\cdot\,]_n$ are the maps given by \eqref{cut_out_operator} with respect to $(e_n)_{n\in\mathbb N}$. Hence, it follows
\begin{align*}
\lim_{n\to\infty} \frac{\operatorname{tr}([T]_{2n})}{2n}=\lim_{n\to\infty}\frac{1}{2n}\sum_{j=1}^{2n}\langle e_j,Te_j\rangle=\mu\,.
\end{align*}
Additionally, by Lemma \ref{embedding_trace_preserv} and \ref{lemma_proj_strong_conv}, one has
\begin{align*}
\lim_{n\to\infty}|\operatorname{tr}(C)-\operatorname{tr}([C]_{2n})|=\lim_{n\to\infty} |\operatorname{tr}(C-C\Pi_{2n})|\leq \lim_{n\to\infty} \nu_1(C-C\Pi_{2n})=0.
\end{align*}
This shows $\operatorname{tr}([C]_{2n})\operatorname{tr}([T]_{2n})/(2n) \to \operatorname{tr}(C)\mu$ for
$n\to\infty$. On the other hand, $W_{[C]_{2n}}([T]_{2n})$ is star-shaped with respect to 
$\operatorname{tr}([C]_{2n})\operatorname{tr}([T]_{2n})/(2n)$ for all $n\in\mathbb N$,
cf.~\cite[Thm. 4]{article_cheungtsing}. This means that the sequence of star-centers
converges to $\operatorname{tr}(C)\mu$ and thus Lemma \ref{lemma_5} (d) and Theorem \ref{lemma_2} imply that $\overline{W_{C}(T)}$ is star-shaped with respect to 
$\operatorname{tr}(C)\mu$. As $\mu\in W_e(T)$ was chosen arbitrarily, the proof is complete.
\end{proof}

\begin{remark}
In finite dimensions, Tsing \cite{tsing_star_shaped} showed that for normal $C\in\mathbb C^{n\times n}$ and arbitrary $A\in\mathbb C^{n\times n}$, $W_C(A)$ is star-shaped with respect to $(\operatorname{tr}(C)\operatorname{tr}(A))/n$. Nine years later Hughes \cite{article_hughes} proved, in our words, that $\overline{W_C(T)}$ is star-shaped with respect to $\operatorname{tr}(C)W_e(T)$ for all normal $C\in\mathcal F(\mathcal H)$ and all $T\in\mathcal B(\mathcal H)$. This was generalized to arbitrary $C\in\mathcal F(\mathcal H)$ by Jones \cite{article_jones} and in finite dimensions to arbitrary $C\in\mathbb C^{n\times n}$ by Cheung and Tsing \cite{article_cheungtsing}. 

However, none of the authors provided a satisfying link between the star-center in finite dimensions
and the set of star-centers in infinite dimensions. The above proof as well as characterization (c) of Proposition \ref{prop_w_e}, which is new to our knowledge, now 
clearly suggest that the set $\operatorname{tr}(C)W_e(T)$ is a natural replacement of
$(\operatorname{tr}(C)\operatorname{tr}(A))/n$ in infinite dimensions.
\end{remark}

\noindent \textbf{Open Problems.}
\begin{itemize}
\item[(a)] The $C$-numerical range of $T \in B(\mathcal H)$ is nothing else than
the range of the bounded linear functional $\ell(\cdot) := \operatorname{tr}(C(\cdot))$ restricted
to the unitary orbit 
$\lbrace U^\dagger TU \,|\,U\in\mathcal B(\mathcal H)\text{ unitary}\rbrace$ of $T$. 
Since it is well known that $\mathcal B^1(\mathcal H)$ (by the above identification)
is only a proper subspace of the dual space $\mathcal B(\mathcal H)'$ of $\mathcal B(\mathcal H)$, it is quite natural to ask whether convexity or star-shapedness of 
\begin{align*}
\lbrace\ell(U^\dagger TU)\,|\,U\in\mathcal B(\mathcal H)\text{ unitary}\rbrace
\end{align*}
holds for \textit{arbitrary} $\ell\in\mathcal B(\mathcal H)'$.\vspace{4pt}
\item[(b)] Westwick \cite{article_westwick} showed, in our words, that for all hermitian $C\in\mathcal F(\mathcal H)$ and all $T\in\mathcal B(\mathcal H)$, the $C$-numerical range $W_C(T)$ is convex (without closure). Thus it is natural to ask whether or not Theorem \ref{theorem_1a} holds if $\overline{W_C(T)}$ is replaced by $W_C(T)$. For Theorem \ref{theorem_1}, we know that it fails if $\overline{W_C(T)}$ is replaced by $W_C(T)$ due to the fact that the set $\operatorname{tr}(C)W_e(T)$ may drop out of $W_C(T)$ (consider e.g.~$C=T=\operatorname{diag}(1/2^n)_{n\in\mathbb N}$ with respect to an arbitrary orthonormal basis, obviously $W_e(T)=\lbrace 0\rbrace$ but $0\notin W_C(T)$ as the respective traces are always positive). However, this of course does not rule out that $W_C(T)$ may be still star-shaped yet with respect to another star-center. 
\end{itemize}

\subsection{The $C$-spectrum}\label{sect_C_spectrum}

The $C$-spectrum is a powerful tool in order to gain further knowledge about the $C$-numerical range
which was first introduced for matrices in \cite{article_marcus}. We want to transfer this concept
and some of the known results to infinite dimensions.

In order to define the $C$-spectrum, we first have to fix the term \emph{eigenvalue sequence} of a 
compact operator $T \in \mathcal K(\mathcal H)$. In general, it is obtained by arranging the 
(necessarily countably many) non-zero eigenvalues in decreasing order with respect to their
absolute values and each eigenvalue is repeated as many times as its algebraic multiplicity\footnote{By
\cite[Prop. 15.12]{MeiseVogt}, every non-zero element $\lambda \in \sigma(T)$ of the spectrum
of $T$ is an eigenvalue of $T$ and has a well-defined finite algebraic multiplicity $\nu_a(\lambda)$,
e.g.~$\nu_a(\lambda) := \dim \ker (T - \lambda I)^{n_0}$,  where $n_0 \in \mathbb N$ is the smallest 
natural number $n \in \mathbb N$ such that $\ker (T - \lambda I)^n = \ker (T - \lambda I)^{n+1}$.
\label{footnote_alg_mult}}. If only finitely many non-vanishing eigenvalues exist, the sequence is filled
up with zeros, see \cite[Ch. 15]{MeiseVogt}. For our purposes, we have to pass to a slightly 
\emph{modified eigenvalue sequence} as follows: 

\begin{itemize}
\item 
If the range of $T$ is infinite-dimensional and the kernel of $T$ finite-dimensional
then put $\operatorname{dim}(\operatorname{ker}T)$ zeros at the beginning of the eigenvalue
sequence of $T$. \vspace{4pt}
\item 
If the range and the kernel of $T$ are infinite-dimensional, mix infinitely many zeros into the
eigenvalue sequence\footnote{Since in Definition \ref{defi_3} arbitrary permutations
will be applied to the modified eigenvalue sequence, we do not need to specify this mixing 
procedure further, cf. also Lemma \ref{Lemma_5b}.} of $T$.\vspace{4pt}
\item
If the range of $T$ is finite-dimensional, leave the eigenvalue sequence of $T$ unchanged. 
\end{itemize}

\begin{definition}[$C$-spectrum]\label{defi_3}
For $C\in\mathcal B^1(\mathcal H)$ with modified eigenvalue sequence $(\gamma_n)_{n\in\mathbb N}$ and 
$T\in\mathcal K(\mathcal H)$ with modified eigenvalue sequence $(\tau_n)_{n\in\mathbb N}$, we define 
the $C$-spectrum of $T$ to be
\begin{align*}
P_C(T)=\Big\lbrace \sum\nolimits_{n=1}^\infty \gamma_n\tau_{\sigma(n)} \,\Big|\, \sigma:\mathbb N \to\mathbb N \text{ is permutation}\Big\rbrace\,.
\end{align*}
\end{definition}

A survey regarding the $C$-spectrum of a matrix can be found in \cite[Ch. 6]{article_li_radii}. Now note that compact normal operators have a spectral decomposition of the form
\begin{align*}
T = \sum_{n=1}^\infty \tau_n \langle f_n, \cdot \rangle f_n
\end{align*}
where $(f_n)_{n \in \mathbb N}$ is an orthonormal basis of $\mathcal H$ and $(\tau_n)_{n \in \mathbb N}$
denotes the modified eigenvalue sequence of $T$, cf.~\cite[Thm. VIII.§4.6]{berberian1976}. 
If an operator is normal but not compact, it still allows a spectral decomposition but, in general,
the above (finite or infinite) sum has to be replaced by an integral which makes the definition
of its $C$-spectrum quite delicate. Therefore, we will restrict our consideration to the compact case.

\begin{theorem}\label{theorem_4}
Let $C\in\mathcal B^1(\mathcal H)$ and $T\in\mathcal K(\mathcal H)$ be both normal. Then one has
\begin{align*}
P_C(T)\subseteq W_C(T)\subseteq\operatorname{conv}(\overline{P_C(T)})\,.
\end{align*}
\end{theorem}
\begin{proof}[Proof of Theorem \ref{theorem_4} -- first inclusion]
Let $(e_n)_{n\in\mathbb N}$ and $(f_n)_{n\in\mathbb N}$ be orthonormal bases of $\mathcal H$
such that $C$ and $T$ can be represented as
\begin{align*}
C = \sum_{n=1}^\infty \gamma_n \langle e_n, \cdot \rangle e_n
\quad\text{and}\quad
T = \sum_{n=1}^\infty \tau_n \langle f_n, \cdot \rangle f_n
\end{align*}
where $(\gamma_n)_{n\in\mathbb N}$ and $(\tau_n)_{n\in\mathbb N}$ are the modified eigenvalue sequence
of $C$ and $T$, respectively. Now let $\sigma:\mathbb N\to\mathbb N$ be any permutation and define
the operator
\begin{align*}
U_\sigma:=\sum_{n=1}^\infty\langle e_n,\cdot\rangle f_{\sigma(n)} \in \mathcal B(\mathcal H)\,.
\end{align*}
Obviously, $U_\sigma$ is unitary by the Fourier expansion and yields the following equality:
\begin{align*}
\operatorname{tr}(CU_\sigma^\dagger TU_\sigma) = \sum_{n=1}^\infty \langle e_n, CU_\sigma^\dagger TU_\sigma e_n\rangle
= \sum_{n=1}^\infty\gamma_n\langle f_{\sigma(n)},T f_{\sigma(n)} \rangle
=\sum_{n=1}^\infty\gamma_n\tau_{\sigma(n)}
\end{align*}
The fact that $\sigma$ was chosen arbitrarily shows the first inclusion.
\end{proof}
\noindent The second inclusion we will prove later as for that, we need some more knowledge of the $C$-spectrum of normal operators.\medskip

For matrices $A,C\in\mathbb C^{n\times n}$, it is well known that the first inclusion
\begin{align*}
P_C(A)\subseteq W_C(A)
\end{align*}
of Theorem \ref{theorem_4} holds even if only $A$ or $C$  is normal \cite[Eq.(4)]{article_marcus}. 
This can be easily seen via Schur's triangularization theorem \cite[Thm. 2.3.1]{Horn_Johnson}. 
In order to generalize this result to operators on i.s.c. Hilbert spaces we recall the following terminology, cf. \cite[Ch. 2]{herrero_triangular}.

\begin{definition}
\begin{itemize}
\item[(a)] An operator $T\in\mathcal B(\mathcal H)$ is called \textit{upper triangular} with respect
to $(e_n)_{n\in\mathbb N}$ if there exists an orthonormal basis $(e_n)_{n\in\mathbb N}$ of $\mathcal H$ such 
that $\langle e_j,Te_k\rangle=0$ for all $j,k\in\mathbb N$ with $j>k$.\vspace{4pt}
\item[(b)] Analogously, $T\in\mathcal B(\mathcal H)$ is \textit{lower triangular} with respect to the
orthonormal basis $(e_n)_{n\in\mathbb N}$ if $\langle e_j,Te_k\rangle=0$ for all $j,k\in\mathbb N$ with $j<k$. 
\end{itemize}
\end{definition}

\begin{theorem}\label{theorem_8}
Let $C\in\mathcal B^1(\mathcal H)$ and $T\in\mathcal K(\mathcal H)$ and assume that one of them
is normal and the other one is upper or lower triangular. Then $\overline{P_C(T)}\subseteq \overline{W_C(T)}$.
\end{theorem}
\begin{proof}
We assume w.l.o.g.~that $T$ is normal and $C$ is upper triangular with respect to the same orthonormal
basis $(e_n)_{n\in\mathbb N}$ which also diagonalizes $T$. Then Theorem
\ref{lemma_12} (see Appendix \ref{appendix_c}) guarantees that there exists a one-to-one correspondence
between the non-zero ``diagonal entries'' $\langle e_n,Ce_n\rangle$ of $C$ and the non-zero elements of 
the modified eigenvalue sequence $(\gamma_j)_{j \in \mathbb N}$ of $C$. Moreover, the non-vanishing singular
values of any compact normal operator are given by the absolute values of its non-zero eigenvalues, 
which guarantees $\sum_{j=1}^\infty |\gamma_j|=\nu_1(C)<\infty$. In addition, the modified eigenvalue
sequence of any compact operator converges to zero and therefore Lemma \ref{Lemma_5b} below
guarantees that one can proceed as in the proof of Theorem \ref{theorem_4} -- first inclusion.
\end{proof}

\begin{lemma}\label{Lemma_5b}
Let $\sigma: \mathbb N \to \mathbb N$ be a permutation and let $(a_n)_{n\in\mathbb N}$, $(b_n)_{n\in\mathbb N}$
be sequences of complex numbers such that $\sum_{j=1}^\infty |a_j|<\infty$ and $(b_n)_{n\in\mathbb N}$ converges
to zero. Moreover, let $(a'_n)_{n\in\mathbb N}$, $(b'_n)_{n\in\mathbb N}$ be sequences of complex numbers
which differ from $(a_n)_{n\in\mathbb N}$, $(b_n)_{n\in\mathbb N}$ only by a finite or infinite number of zeros.
More presicely, for each $\alpha\neq 0$ one has
\begin{align}\label{eq:Lemma_5b_1}
|\{k \in \mathbb N \,|\, a_k = \alpha\}| = |\{k \in \mathbb N \,|\, a'_k = \alpha\}|
\end{align}
and similarly for $(b_n)_{n\in\mathbb N}$ and $(b'_n)_{n\in\mathbb N}$. Then the closures of the following
two sets coincide:
\begin{align*}
A&:=\Big\lbrace \sum\nolimits_{n=1}^\infty a_nb_{\sigma(n)} \,\Big|\, \sigma:\mathbb N \to\mathbb N \text{ is permutation}\Big\rbrace\\
and \quad\quad\quad& \\
A'&:=\Big\lbrace \sum\nolimits_{n=1}^\infty a'_nb'_{\sigma(n)} \,\Big|\, \sigma:\mathbb N \to\mathbb N \text{ is permutation}\Big\rbrace
\end{align*}
\end{lemma}

\noindent
For a proof of Lemma \ref{Lemma_5b} we refer to Appendix \ref{appendix_b}.
\begin{lemma}\label{lemma_6b}
Let $C\in\mathcal B^1(\mathcal H)$ and $T\in\mathcal K(\mathcal H)$ be both normal. Then for all $\varepsilon>0$ and $w\in \overline{P_C(T)}$ there exists $N\in\mathbb N$ such that the distance $d(w,P_{[C]^e_n}([T]^g_n))<\varepsilon$ for all $n\geq N$. Here, $[\,\cdot\,]_n^e$ and $[\,\cdot\,]_n^g$ are the maps given by (\ref{cut_out_operator}) with 
respect to the orthonormal bases $(e_n)_{n\in\mathbb N}$ and $(g_n)_{n\in\mathbb N}$ which diagonalize $C$ and
$T$, respectively.
\end{lemma}
\begin{proof}
Let $\varepsilon>0$ and $w\in\overline{P_C(T)}$ be given. There exists a permutation 
$\sigma: \mathbb N \to \mathbb N$ with
\begin{align*}
\Big| w-\sum_{j=1}^\infty \gamma_j \tau_{\sigma(j)} \Big|<\frac{\varepsilon}{2}\,.
\end{align*}
Furthermore, there exists $N' \in\mathbb N$
such that
\begin{align*}
\sum_{j=N'+1}^\infty |\gamma_j| < \frac{\varepsilon}{4 \|T\|}\,.
\end{align*}
Here we used the fact that the non-vanishing singular values of a compact normal operator coincide
with the absolute values of its non-zero eigenvalues. This guarantees
$\sum_{j=1}^\infty |\gamma_j| \, |\tau_{\sigma(j)}| \leq \nu_1(C)\Vert T\Vert<\infty$. Next, we define
\begin{align*}
N := \max_{1 \leq j \leq N'} \sigma(j)\,.
\end{align*}
Note $N \geq N'$. Hence we can choose a permuation $\sigma':\mathbb N \to \mathbb N$ such that
$\sigma'$ restricted to $\{1,\dots,N'\}$ coincides with $\sigma$ and $\sigma'(j) := j$ for $j > N$.
Then $w_n := \sum_{j=1}^n \gamma_j\tau_{\sigma'(j)}$ belongs to $P_{[C]^e_n}([T]^g_n)$ for all $n \geq N$ as $\lbrace\sigma'(1),\ldots,\sigma'(n)\rbrace=\lbrace1,\ldots,n\rbrace$
and we obtain
\begin{align*}
| w-w_n| & \leq \Big| w- \sum_{j=1}^\infty \gamma_j \tau_{\sigma(j)} \Big| + 
\Big| \sum_{j=1}^\infty \gamma_j \tau_{\sigma(j)} - w_n\Big|\\
&< \frac{\varepsilon}{2} + \sum_{j=N'+1}^\infty |\gamma_j| \, |\tau_{\sigma(j)}|
+ \sum_{j=N'+1}^n |\gamma_j| \, |\tau_{\sigma'(j)}|
<\frac{\varepsilon}{2}+\frac{\varepsilon}{4}+\frac{\varepsilon}{4}=\varepsilon
\end{align*}
for all $n\geq N$. 
\end{proof}

\noindent Note that in the above proof, $N$ depends usually on $\varepsilon$ but also on the chosen 
point $w \in\overline{P_C(T)}$.

\begin{theorem}\label{lemma_6}
Let $C\in\mathcal B^1(\mathcal H)$ and $T\in\mathcal K(\mathcal H)$ be both normal. Then
\begin{align*}
\lim_{n\to\infty}P_{[C]^e_n}([T]^g_n)= \overline{P_C(T)}\,.
\end{align*}
Here, $[\,\cdot\,]_n^e$ and $[\,\cdot\,]_n^g$ are the maps given by (\ref{cut_out_operator}) with 
respect to the orthonormal bases $(e_n)_{n\in\mathbb N}$ and $(g_n)_{n\in\mathbb N}$ which diagonalize $C$ and
$T$, respectively.
\end{theorem}
\begin{proof}
Again, in order to apply the results of Subsection \ref{subsec:set-convergence}, we have to
check that all sets occurring in Theorem \ref{lemma_6} are non-empty and compact. But this
is obviously the case, as all $P_{[C]^e_n}([T]^g_n)$ are non-empty and finite and $\overline{P_C(T)}$ is non-empty,
closed and bounded by $\nu_1(C)\Vert T\Vert$.

Let $(\gamma_j)_{j\in\mathbb N}$ and $(\tau_j)_{j\in\mathbb N}$ denote the modified eigenvalue sequences of 
$C$ and $T$, respectively. Obviously, for arbitrary $n\in\mathbb N$, the eigenvalues of $[C]^e_n$ and $[T]^g_n$
are given by $\lbrace \gamma_1,\ldots,\gamma_n\rbrace$ and $\lbrace \tau_1,\ldots,\tau_n\rbrace$.\medskip

W.l.o.g.~$T \neq 0$. Let $\varepsilon>0$. Due to compactness, there exist finitely many $w_1,\ldots,w_L\in\overline{P_C(T)}$ such that
\begin{align*}
\bigcup_{k=1}^L B_{\varepsilon/2}(w_k)\supset\overline{P_C(T)}
\end{align*}
where $B_{\varepsilon/2}(w_k)$ denotes open $\varepsilon/2$-balls around $w_k$. By Lemma
\ref{lemma_6b}, each of these $w_k$ admits $N_k\in\mathbb N$ such that $d(w_k,P_{[C]^e_n}([T]^g_n))<\varepsilon/2$ for all $n\geq N_k$. Define $N':=\max\lbrace N_1,\ldots,N_L\rbrace$. Now for any $w\in \overline{P_C(T)}$, there exists $k\in\lbrace 1,\ldots,L\rbrace$ such that $|w-w_k|<\varepsilon/2$ and thus
\begin{align*}
d(w,P_{[C]^e_{n}}([T]^g_{n}))\leq |w-w_k|+d(w_k,P_{[C]^e_{n}}([T]^g_{n}))<\varepsilon
\end{align*}
for all $n\geq N'$.

\medskip
Conversely, as in the previous proof there exists $N''$ such that 
such that
\begin{align*}
\sum_{j=N''+1}^\infty |\gamma_j| < \frac{\varepsilon}{\|T\|}\,.
\end{align*}
Let $v_n \in P_{[C]^e_n}([T]^g_n)$ so there exists a permutation
$\sigma_n \in S_n$ such that  
\begin{align*}
v_n=\sum_{j=1}^n \gamma_j\tau_{\sigma_n(j)}\,.
\end{align*}
Obviously, we can extend $\sigma_n$ to a permuation $\tilde\sigma_n:\mathbb N\to\mathbb N$ via
\begin{align*}
\tilde\sigma_n(j):=\begin{cases}\sigma_n(j)&1\leq j\leq n\,,\\ j&j>n\,.\end{cases}
\end{align*}
Then for $\tilde v_n := \sum_{j=1}^\infty\gamma_j\tau_{\tilde\sigma_n(j)} \in P_C(T) \subseteq\overline{P_C(T)}$ one has
\begin{align*}
|v_n-\tilde v_n| &= \Big|\sum_{j=1}^n \gamma_j\tau_{\sigma_n(j)} - \sum_{j=1}^\infty\gamma_j\tau_{\tilde\sigma_n(j)} \Big|\\
& = \Big|\sum_{j=n+1}^\infty\gamma_j\tau_j \Big| \leq \sum_{j=N+1}^\infty|\gamma_j||\tau_j| \leq 
\Vert T\Vert \sum_{j=N+1}^\infty|\gamma_j| <\varepsilon
\end{align*}
which yields $d(v_n,\overline{P_C(T)})<\varepsilon$ for all $n\geq N''$. Thus, choosing $N:=\max\lbrace N',N''\rbrace$, Lemma \ref{lemma_11} implies that the Hausdorff distance $\Delta(P_{[C]^e_n}([T]^g_n),\overline{P_C(T)})<\varepsilon$ for all $n\geq N$.
\end{proof}

With this result at hand, we can finally come back to the remaining part of the proof of Theorem \ref{theorem_4}.

\begin{proof}[Proof of Theorem \ref{theorem_4} -- second inclusion]
Let $(e_n)_{n\in\mathbb N}$
and $(g_n)_{n\in\mathbb N}$ be the orthonormal bases of $\mathcal H$ which diagonalize $C$ and $T$, respectively. Furthermore, let $[\,\cdot\,]_n^e$ and $[\,\cdot\,]_n^g$ be the maps given by \eqref{cut_out_operator} with respect to $(e_n)_{n\in\mathbb N}$ and $(g_n)_{n\in\mathbb N}$, respectively. By assumption, $[C]^e_k$ and $[T]^g_k$ are diagonal hence normal for all $k\in\mathbb N$, so by \cite[Corollary 2.4]{article_sunder} this implies
\begin{align}\label{eq:wc_pc_incl}
W_{[C]_{2n}^e}([T]_{2n}^g)\subseteq \operatorname{conv}(P_{[C]_{2n}^e}([T]_{2n}^g))
\end{align}
for all $n\in\mathbb N$. Using Lemma \ref{lemma_5}, Theorem \ref{lemma_2} and \ref{lemma_6}, we conclude
\begin{align*}
W_C(T)\subseteq \overline{W_C(T)}=\lim_{n\to\infty} W_{[C]_{2n}^e}([T]_{2n}^g)\subseteq \lim_{n\to\infty} \operatorname{conv}\big( P_{[C]_{2n}^e}([T]_{2n}^g) \big)= \operatorname{conv}(\overline{P_C(T)})\,.
\end{align*}
\end{proof}
\noindent Another proof of the second inclusion of Theorem \ref{theorem_4}, which is more oriented along the lines of the original proof \cite[Corollary 2.4]{article_sunder} can be found in Appendix \ref{app_sec_incl}. \medskip

In finite dimensions, it is well known \cite[Thm. 4]{article_marcus}, that for $A,C\in\mathbb C^{n\times n}$ 
one has
\begin{align}\label{lemma_marvin_conv}
W_C(A)=\operatorname{conv}(P_C(A))\,,
\end{align}
whenever $A$ and $C$ are both normal and the eigenvalues of $C$ form a collinear set in the complex plane. A generalization of this result to i.s.c. Hilbert spaces reads as follows.

\begin{corollary}\label{theorem_3}
Let $C\in\mathcal B^1(\mathcal H)$ and $T\in\mathcal K(\mathcal H)$ be both normal and assume that
the eigenvalues of $C$ or $T$ are collinear. Then
\begin{align*}
\overline{W_C(T)}=\operatorname{conv}(\overline{P_C(T)})\,.
\end{align*}
\end{corollary}
\begin{proof}
By Theorem \ref{theorem_4} one has $P_C(T)\subseteq W_C(T)\subseteq\operatorname{conv}(\overline{P_C(T)})$. Hence, taking the closure and convex hull yields
\begin{align*}
\operatorname{conv}(\overline{W_C(T)})=\operatorname{conv}(\overline{P_C(T)})\,.
\end{align*}
Here, we used the fact that the convex hull of a compact set in $\mathbb{R}^n$ is again compact.
On the other hand, $C$ meets the conditions of Theorem \ref{theorem_1a} and thus $\overline{W_C(T)}$
is already convex. This yields the desired equality and concludes the proof.
\end{proof}
\noindent The above proof was suggested by the referee and provides a major simplification of our original proof.
\section{Appendix}\label{appendix}
\renewcommand{\thesubsection}{\Alph{subsection}}

\subsection{Proof of Lemma \ref{lemma_5}}\label{appendix_A}
Note that for bounded sequences $(A_n)_{n\in\mathbb N}$ of non-empty
compact subsets of $\mathbb C$ which converges to $A$ with respect to the Hausdorff metric
one has the following characterization of the limit set according to Lemma \ref{lem:Hausdorff}:
\begin{align*}
x \in A \;\Longleftrightarrow\;\text{there exists a sequence $(a_n)_{n\in\mathbb N}$ with 
$a_n \in A_n$ and $a_n \to x$ for $n \to \infty$\,.}
\end{align*}

\begin{proof}[Proof of Lemma \ref{lemma_5}]
(a) Let $x\in A$ be given. Then there exists a sequence $(a_n)_{n\in\mathbb N}$ with 
$a_n \in A_n$ and $a_n \to x$ for $n \to \infty$. By assumption, we have $A_n \subset B_n$
and thus $a_n \in B_n$. Hence, by the above characterization of the limit set we obtain $x \in B$.

\medskip
\noindent
(b) Let $\varepsilon>0$ be given. By assumption there exists $N \in\mathbb  N$ such 
that for all $n \geq N$, $\Delta(A_n,A)<\varepsilon$. By Lemma \ref{lemma_11}, the latter 
is equivalent to the assertion that for all $a\in A$ there exists $a_n \in A_n$ satisfying
$|a-a_n|<\varepsilon$ and for all $a'_n \in A_n$ there exists $a' \in A$ with $|a'-a'_n|<\varepsilon$. 

\medskip
\noindent
First, let $x \in \operatorname{conv}(A)$ be arbitrary. By Caratheodory's theorem,
$x \in \operatorname{conv}(A)$ can be written as
\begin{align*}
x = r a + s b + t c 
\end{align*}
with $a,b,c\in A$, $r,s,t \geq 0$, and $r + s + t = 1$. 
Then for all $n\geq N$ we can choose $a_n,b_n,c_n\in A_n$ with distance less than 
$\varepsilon$ to $a,b,c$, respectively. This yields for
\begin{align*}
x_n := r a_n + s b_n + t c_n \in \operatorname{conv}(A_n)
\end{align*}
the estimate
\begin{align*}
|x-x_n|\leq r|a-a_n| + s|b-b_n| + t|c-c_n| < \varepsilon\,.
\end{align*}
Similarly, for every $x'_n\in\operatorname{conv}(A_n)$ one can choose $x'\in\operatorname{conv}(A)$ 
with $|x'-x'_n|<\varepsilon$ for all $n\geq N$. This proves (b) according to Lemma \ref{lemma_11}.

\medskip
\noindent
(c) If $A_n$ is convex, one has $A_n = \operatorname{conv}(A_n)$ for all $n\in\mathbb N$ and 
therefore by (b) we immediately obtain
\begin{align*}
A=\lim_{n\to\infty}A_n=\lim_{n\to\infty}\operatorname{conv}(A_n)=\operatorname{conv}(A)\,.
\end{align*}
Hence, $A$ is convex.

\medskip
\noindent
(d) We have to show $t z + (1-t)a \in A$ for all $a \in A$ and $t \in[0,1]$. To this end, let 
$a \in A$ and choose $a_n \in A_n$ such that $a_n \to a$ for $n \to \infty$. Since $A_n$ is
star-shaped with respect to $z_n$ one has $t z_n + (1-t)a_n \in A_n$ for all $n \in \mathbb N$.
Moreover, $t z_n + (1-t)a_n$ converges obviously to $t z + (1-t)a$ and therefore by the 
above characterization of the limit set we conclude $t z + (1-t)a \in A$.
\end{proof}

\subsection{Proof of Lemma \ref{U-approximation}}\label{App_Lemma_U}

To prove Lemma \ref{U-approximation} we need the following auxiliary result.

\begin{lemma}\label{lem:unitary-dilation}
Let $U\in\mathbb C^{n\times n}$ with $\Vert U\Vert\leq 1$. Then one can find matrices
$Q,R,S\in\mathbb C^{n\times n}$ such that
\begin{align*}
V:=\begin{pmatrix} U&Q\\R&S \end{pmatrix}\in\mathbb C^{2n\times 2n}
\end{align*}
is unitary.
\end{lemma}

\begin{proof}
Obviously, $\Vert U\Vert\leq 1$ implies $\operatorname{I}_n-U U^\dagger\geq 0$, where $\operatorname{I}_n$ denotes 
the $n\times n$ identity matrix. Hence 
$Q:=\sqrt{\operatorname{I}_n-UU^\dagger }$ is well-defined. Now the upper $n$ rows of $V$ form an orthonormal system in 
$\mathbb C^{2n}$ as
\begin{align*}
\begin{pmatrix} U&Q \end{pmatrix}\begin{pmatrix} U^\dagger\\Q^\dagger \end{pmatrix}
= UU^\dagger +QQ^\dagger=\operatorname{I}_n.
\end{align*}
Completing this orthonormal system to an orthonormal basis of $\mathbb C^{2n}$ gives $R,S$ such that,
in total, $V$ is unitary.
\end{proof}

\begin{proof}[Proof of Lemma \ref{U-approximation}]
Let $U\in\mathcal B(\mathcal H)$ be unitary and consider arbitrary orthonormal bases 
$(e_n)_{n\in\mathbb N}$, $(g_n)_{n\in\mathbb N}$ of $\mathcal H$. For all $n\in\mathbb N$ one has $\Vert (\Gamma_n^g)^\dagger U\Gamma_n^e\Vert\leq 1$ so Lemma
\ref{lem:unitary-dilation} yields $Q_n,R_n,S_n\in\mathbb C^{n\times n}$ such that
\begin{align*}
V_n:=\begin{pmatrix} (\Gamma_n^g)^\dagger U\Gamma_n^e&Q_n\\R_n&S_n \end{pmatrix}\in\mathbb C^{2n\times 2n}
\end{align*}
is unitary. Define $\hat U_n:=\Gamma_{2n}^gV_n(\Gamma_{2n}^e)^\dagger\in\mathcal B(\mathcal H)$. Then, obviously, (b) and (c) 
of Lemma \ref{U-approximation} hold. To show that $(\hat U_n)_{n\in\mathbb N}$ converges strongly to
$U$ we first observe $\|\hat U_n x - Ux\| \leq \|\hat U_n x - \Pi^g_{n}U\Pi^e_{n}x\| + \|\Pi^g_{n}U\Pi^e_{n}x - Ux\|$
and 
\begin{align*}
\|\Pi^g_{n}U\Pi^e_{n}x - Ux\| \leq\Vert \Pi^g_{n}U\Pi^e_{n}x-\Pi^g_nUx\Vert+ \Vert\Pi^g_nUx-Ux\Vert\leq \Vert \Pi^e_{n}x-x\Vert+ \Vert\Pi^g_nUx-Ux\Vert \,.
\end{align*}
Hence, $(\Pi^g_{n}U\Pi^e_{n})_{n\in\mathbb N}$ converges strongly to $U$ by Lemma \ref{lemma_proj_strong_conv} (a)
and, therefore, it suffices to show that
\begin{align*}
Z_n:=\hat U_n-\Pi^g_{n}U\Pi^e_{n}=\Gamma_{2n}^g\begin{pmatrix} 0&Q_n\\R_n&S_n \end{pmatrix}(\Gamma_{2n}^e)^\dagger
\end{align*}
strongly converges to 0. Let $x \in\mathcal H\setminus\lbrace 0\rbrace$ and $\varepsilon>0$ be given.
Again, by Lemma \ref{lemma_proj_strong_conv} (a), one can choose $N\in\mathbb N$ such that
\begin{align}
\begin{split}
\Vert x\Vert^2-\Vert\Pi_n^gUx\Vert^2 &= \Vert Ux\Vert^2-\Vert\Pi_n^gUx\Vert^2 = \Vert\Pi_n^gUx-Ux\Vert^2<\frac{\varepsilon^2}{8}\\
\text{and} \quad \Vert\Pi_n^e x-x\Vert&<\min\Big\lbrace\frac{\varepsilon^2}{16\Vert x\Vert},\frac{\varepsilon}{2\sqrt{2}}\Big\rbrace\end{split}\label{pi_ineqs}
\end{align}
for all $n\geq N$. Now let $\Lambda_n^e:\mathbb C^n\to\mathcal H$ be the unique linear operator
given by $\hat e_j\mapsto e_{j+n}$ for $j\in\lbrace1,\ldots,n\rbrace$. So basically $(\Lambda_n^e)^\dagger$ ``cuts out'' the components $x_{n+1},\ldots,x_{2n}$ of $x\in\mathcal H$ with respect to $(e_n)_{n\in\mathbb N}$. Next, we decompose $x$ as follows
\begin{align*}
x=\Pi_n^ex+(\Pi^e_{2n}-\Pi^e_n)x+(\operatorname{id}_{\mathcal H}-\Pi_{2n}^e)x\,.
\end{align*}
Then $\Pi_n^ex \in \mathcal H$ and $x_n:=(\Gamma_n^e)^\dagger x\in\mathbb C^n$ are essentially the same vectors, as those differ only by the isometric embedding $\Gamma_n^e$. The same holds for $(\Pi^e_{2n}-\Pi^e_n)x \in \mathcal H$ and $y_n:=(\Lambda_n^e)^\dagger x\in\mathbb C^n$. 
Taking into account that $\Gamma_{2n}^g$ is an isometry, we obtain
\begin{align*}
\Vert x\Vert^2\geq\Vert\hat U_nx\Vert^2=\Vert(\Gamma_n^g)^\dagger U\Gamma_n^ex_n+Q_ny_n\Vert^2+\Vert R_nx_n+S_ny_n\Vert^2
\end{align*}
and thus
\begin{align}\label{ineqs-2}
\Vert R_nx_n+S_ny_n\Vert^2&\leq\Vert x\Vert^2-\Vert(\Gamma_n^g)^\dagger U\Gamma_n^ex_n+Q_ny_n\Vert^2\nonumber\\
&= \Vert x\Vert^2-\Vert(\Gamma_n^g)^\dagger Ux - ((\Gamma_n^g)^\dagger Ux-(\Gamma_n^g)^\dagger U\Gamma_n^ex_n-Q_ny_n)\Vert^2\nonumber\\
&\leq \Vert x\Vert^2 - \big|\Vert(\Gamma_n^g)^\dagger Ux \Vert - \Vert (\Gamma_n^g)^\dagger Ux-(\Gamma_n^g)^\dagger U\Gamma_n^ex_n-Q_ny_n\Vert \big|^2 \,,
\end{align}
where the last estimate follows from the reverse triangle inequality. Then, using again that $\Gamma_n^g$ is an isometry satisfying $\Gamma_n^g(\Gamma_n^g)^\dagger=\Pi_n^g$ and further $\Vert Q_n\Vert\leq 1$ by construction, we deduce from 
\eqref{pi_ineqs} and \eqref{ineqs-2} the estimate
\begin{align*}
\Vert R_nx_n+S_ny_n\Vert^2&\leq\Vert x\Vert^2-\Vert\Pi_n^g Ux\Vert^2+2\Vert \Pi_n^g Ux \Vert\Vert\Pi_n^gUx-\Pi_n^gU\Pi_n^ex-\Gamma_n^gQ_ny_n\Vert\\
&<\frac{\varepsilon^2}{8}+2\Vert\Pi_n^g Ux\Vert \big(\Vert\Pi_n^gU\Vert\Vert x-\Pi_n^ex\Vert+\Vert Q_n\Vert\Vert\Pi_{2n}^ex-\Pi_n^ex\Vert\big)\\
&\leq\frac{\varepsilon^2}{8}+2\Vert x\Vert\big(\Vert x-\Pi_n^ex\Vert+\Vert \Pi_{2n}^ex-x\Vert+\Vert x-\Pi_n^ex\Vert\big)
< \frac{\varepsilon^2}{2}
\end{align*}
for all $n\geq N$. Finally, it follows
\begin{align*}
\Vert Z_nx\Vert^2&=\Vert Q_ny_n\Vert^2+\Vert R_nx_n+S_ny_n\Vert^2<\Vert Q_n\Vert^2\Vert\Pi_{2n}^ex-\Pi_n^ex\Vert^2+\frac{\varepsilon^2}{2}\\
&\leq  (\Vert \Pi_{2n}^ex-x\Vert+\Vert x-\Pi_n^ex\Vert)^2+\frac{\varepsilon^2}{2}<\Big( 2\frac{\varepsilon}{2\sqrt{2}} \Big)^2+\frac{\varepsilon^2}{2}=\varepsilon^2
\end{align*}
for all $n\geq N$. This proves part (a) and, in total, Lemma \ref{U-approximation}.
\end{proof}

\subsection{Proof of Lemma \ref{Lemma_5b}}\label{appendix_b}

Recall that $(a_n)_{n\in\mathbb N}$ and $(b_n)_{n\in\mathbb N}$ are sequences of complex numbers
such that $\sum_{j=1}^\infty |a_j|<\infty$ and $(b_n)_{n\in\mathbb N}$ converges to zero while
$(a'_n)_{n\in\mathbb N}$ and $(b'_n)_{n\in\mathbb N}$ are sequences of complex numbers which differ 
from $(a_n)_{n\in\mathbb N}$ and $(b_n)_{n\in\mathbb N}$, respectively, only by a finite or infinite
number of zeros.

\begin{proof}[Proof of Lemma \ref{Lemma_5b}]
Consider the following intermediate sets:
\begin{align*}
A&:=\Big\lbrace \sum\nolimits_{n=1}^\infty a_nb_{\sigma(n)} \,\Big|\, \sigma:\mathbb N \to\mathbb N \text{ is permutation}\Big\rbrace\,,\\
A_1&:=\Big\lbrace \sum\nolimits_{n=1}^\infty a_nb'_{\sigma(n)} \,\Big|\, \sigma:\mathbb N \to\mathbb N \text{ is permutation}\Big\rbrace\,,\\
A_2&:=\Big\lbrace \sum\nolimits_{n=1}^\infty a'_nb_{\sigma(n)} \,\Big|\, \sigma:\mathbb N \to\mathbb N \text{ is permutation}\Big\rbrace\,,\\
A'&:=\Big\lbrace \sum\nolimits_{n=1}^\infty a'_nb'_{\sigma(n)} \,\Big|\, \sigma:\mathbb N \to\mathbb N \text{ is permutation}\Big\rbrace\,.
\end{align*}
We will proceed as follows: First we will show that the closure of $A$ and $A_1$ coincides, then 
that of $A$ and $A_2$ and finally that of $A_2$ and $A'$. In doing so, we can assume w.l.o.g.~that 
$(a_n)_{n\in\mathbb N}$ does not vanish everywhere and thus one has $s:= \sum_{j=1}^\infty |a_j| > 0$.
As $(b_n)_{n\in\mathbb N}$ is a null sequence there exists $\kappa>0$ such that $|b_k|\leq \kappa$ for
all $k\in\mathbb N$. Furthermore, for every $\varepsilon>0$ there exists $N\in\mathbb N$ such that
\begin{align*}
\sum_{j=N+1}^\infty |a_j| < \frac{\varepsilon}{4\kappa} \quad\text{and}\quad 
\max_{n\geq N}\lbrace|b_n|,|b_n'|\rbrace<\frac{\varepsilon}{4 s}\,.
\end{align*}

\medskip
\noindent
To prove $\overline A=\overline{A}_1$ let $\varepsilon>0$ and $x\in\overline A$. Hence there exists a 
permutation $\sigma: \mathbb N \to \mathbb N$ such that $x':=\sum_{n=1}^\infty a_nb_{\sigma(n)}$ satisfies 
$|x-x'| < \varepsilon/4$. Now by \eqref{eq:Lemma_5b_1} one can construct a permutation 
$\hat\sigma: \mathbb N \to \mathbb N$ which satisfies for all $k\in\lbrace1,\ldots,N\rbrace$ the
following properties:
\begin{itemize}
\item If $b_{\sigma(k)}\neq 0$, then $b_{\sigma(k)}=b'_{\hat\sigma(k)}$.\vspace{4pt}
\item If $b_{\sigma(k)}= 0$, then $\hat\sigma(k)\geq N$.
\end{itemize}
Then, for $y:=\sum_{n=1}^\infty a_n b'_{\hat\sigma(n)} \in A_1$ one has 
\begin{align*}
|x-y| & < \frac{\varepsilon}{4}+\Big|\sum_{n=1}^N a_n(b_{\sigma(n)}-b'_{\hat\sigma(n)})\Big|
+ \sum_{n=N+1}^\infty |a_n|\big(|b_{\sigma(n)}|+|b'_{\hat\sigma(n)}|\big)\\
& <\frac{3\varepsilon}{4} + \max_{n\in\lbrace1,\ldots,N\rbrace}|b_{\sigma(n)}-b'_{\hat\sigma(n)}|\sum_{n=1}^N|a_n|
< \frac{3\varepsilon}{4} + \frac{\varepsilon}{4 s} \sum_{n=1}^N|a_n|\leq\varepsilon\,.
\end{align*}
This shows the inclusion $\overline A \subset \overline{A}_1$. Obviously, the role of
$(b_n)_{n \in \mathbb N}$ and $(b_n')_{n \in \mathbb N}$ is interchangeable and thus the converse
inclusion follows in the same way.

\medskip
\noindent
Next, we prove $\overline A = \overline{A}_2$. As by assumption all sums converge absolutely, 
rearranging them via permutations does not change their value and thus
\begin{align*}
A=\Big\lbrace \sum\nolimits_{n=1}^\infty a_{\sigma(n)}b_{n} \,\Big|\, \sigma:\mathbb N \to\mathbb N \text{ is permutation}\Big\rbrace
\end{align*}
and analogously for $A_2$. Now let $\varepsilon>0$ and $x\in \overline{A}$. Then there exists a permutation 
$\sigma: \mathbb N \to \mathbb N$ such that $x'' := \sum_{n=1}^\infty a_{\sigma(n)}b_n$ satisfies 
$|x-x''| < \varepsilon/4$. Furthermore, one can choose $N' \geq N$ such that 
\begin{align*}
|a_n'|<\frac{\varepsilon}{4\kappa N}
\end{align*}
for all $n\geq N'\geq N$. Again, due to \eqref{eq:Lemma_5b_1} one can construct a permutation
$\hat\sigma: \mathbb N \to \mathbb N$ which satisfies for all $k\in\lbrace1,\ldots,N\rbrace$ the following:
\begin{itemize}
\item If $a_{\sigma(k)}\neq 0$, then $a_{\sigma(k)}=a'_{\hat\sigma(k)}$.\vspace{4pt}
\item If $a_{\sigma(k)}= 0$, then $\hat\sigma(k)\geq N'$.
\end{itemize}
Hence, for $y:=\sum_{n=1}^\infty a'_{\hat\sigma(n)} b_n\in A_2$ we obtain
\begin{align*}
|x-y|&<\frac{\varepsilon}{4} + \kappa\sum_{n=1}^N |a_{\sigma(n)}-a'_{\hat\sigma(n)}|
+ \max_{n\geq N}|b_n|\Big(\sum_{n=N+1}^\infty |a_{\sigma(n)}|+\sum_{n=N+1}^\infty |a'_{\hat\sigma(n)}|\Big)\\
&< \frac{\varepsilon}{4} + \frac{\kappa\varepsilon}{4 \kappa N}N 
+ \frac{\varepsilon}{2 s}\sum_{n=1}^\infty |a_n|= \varepsilon\,.
\end{align*}
Here we used $\sum_{n=1}^\infty |a_n|=\sum_{n=1}^\infty |a_n'|$ as implied by \eqref{eq:Lemma_5b_1}. 
As before, we can interchange the role of $(a_n)_{n \in \mathbb N}$ and $(a_n')_{n \in \mathbb N}$
and thus conclude $\overline A=\overline{A}_2$. 

\medskip
\noindent
Finally, $\overline{A}=\overline{A}_1$ implies $\overline{A}_2 = \overline{A'}$ by choosing
$(a'_n)_{n\in\mathbb N}=(a_n)_{n\in\mathbb N}$ and therefore $\overline{A}=\overline{A}_2=\overline{A'}$.
\end{proof}

\noindent
Note that Lemma \ref{Lemma_5b} (a) becomes false if one waives the assumption that $(b_n)_{n\in\mathbb N}$
converges to zero. For an example, see Appendix \ref{appendix_examples} (Ex. \ref{ex_3}).

\subsection{The spectrum of compact triangular operators}\label{appendix_c}
\begin{theorem}\label{lemma_12}
Let $T\in\mathcal K(\mathcal H)$ be upper or lower triangular with respect to
the orthonormal basis $(e_n)_{n\in\mathbb N}$. Then
\begin{align*}
\sigma(T) \setminus \{0\} = \lbrace\langle e_j,Te_j\rangle\,|\,j\in\mathbb N\rbrace \setminus \{0\}\,.
\end{align*}
Moreover, for all non-zero $\lambda \in \sigma(T)$ the algebraic multiplicity $\nu_a(\lambda)\in\mathbb N$
coincides with the cardinality of the set $\lbrace j\in\mathbb N \,|\, \lambda = \langle e_j,Te_j\rangle\rbrace$.
\end{theorem}

\begin{proof}
First, let us assume $T\in\mathcal K(\mathcal H)$ to be upper triangular with respect to the orthonormal
basis $(e_n)_{n\in\mathbb N}$ and define $\mathcal H_n$ to be the linear span of $e_1, \dots, e_n$. Note
that each $\mathcal H_n$ is a finite-dimensional invariant subspace of $T$. 
Now consider any nonzero $\lambda \in \mathbb C$. According to Lemma \ref{lemma_proj_strong_conv} 
one can choose $n \in \mathbb N$ such that $\|T_n - T\| < |\lambda|$, where $T_n$ denotes
the corresponding block approximation \eqref{eq:defi_block} of $T$ with respect to $(e_n)_{n\in\mathbb N}$. Then the orthogonal
decomposition $\mathcal H = \mathcal H_n \oplus \mathcal H_n^\perp$ induces the following block
matrix representations
\begin{align*}
T := \begin{pmatrix}
A & B \\ 0 & C
\end{pmatrix}
\quad \text{and}\quad
T_n := \begin{pmatrix}
A & 0 \\ 0 & 0
\end{pmatrix}
\end{align*}
of $T$ and $T_n$ with $\|C\| < |\lambda|$, where $A$ and $C$ are upper triangular. Hence one 
has the following equivalences:
\begin{align*}
T - \lambda \operatorname{id}_{\mathcal H} \;\text{is invertible}
\quad\Longleftrightarrow\quad
\begin{pmatrix}
A - \lambda\operatorname{id}_{\mathcal H_n} & B \\ 0 & C - \lambda \operatorname{id}_{\mathcal H_n^{\perp}}
\end{pmatrix}
\;\text{is invertible}\\[2mm]
\Longleftrightarrow\quad
A - \lambda\operatorname{id}_{\mathcal H_n} \;\text{is invertible}
\quad\Longleftrightarrow\quad
\lambda \neq  \langle e_j,Te_j\rangle \;\text{for all}\; j = 1, \dots, n\,.
\end{align*}
Therefore, we conclude
$\sigma(T) \setminus \{0\} = \lbrace\langle e_j,Te_j\rangle\,|\,j\in\mathbb N\rbrace \setminus \{0\}$.
Moreover, because of the straightforward equivalence:
\begin{align*}
x \in \ker (T - \lambda \operatorname{id}_{\mathcal H})^n
\quad\Longleftrightarrow\quad
\Pi_n x \in \ker (A - \lambda \operatorname{id}_{\mathcal H})^n \;\text{and}\; (\operatorname{id}_{\mathcal H}-\Pi_n)x = 0\,,
\end{align*}
where $\Pi_n$ denotes as usual the corresponding orthogonal projection, the algebraic multiplicity (cf. footnote \ref{footnote_alg_mult}) of $\lambda \neq 0$ with respect to
$T$ is equal to the algebraic multiplicity of $\lambda \neq 0$ with respect to $A$, which obviously
equals the number of diagonal entries of $A$ that coincide with $\lambda \neq 0$. 

\medskip
\noindent
Finally, if $T\in\mathcal K(\mathcal H)$ is lower triangular with respect to the orthonormal
basis $(e_n)_{n\in\mathbb N}$, we simply pass to $T^\dagger$ which now obviously is upper triangular with respect to the same basis. Then, keeping in mind the following facts,
the result follows immediately from the first part:
\begin{itemize}
\item
$\sigma(T^\dagger) = \overline{\sigma(T})$, where $\overline{(\cdot)}$ denotes the complex
conjugate.\vspace{4pt}
\item $T^\dagger$ is compact if and only if $T$ is compact and for all $\lambda \neq 0$, the algebraic
multiplicity of $\lambda$ with respect to $T$ coincides with the algebraic multiplicity of
$\overline{\lambda}$ with respect to $T^\dagger$ as a simple consequence of \cite[Lemma 15.9 \& 15.10]{MeiseVogt}.\vspace{4pt}
\item
$\overline{\langle e_j,T^\dagger e_j\rangle} = \langle e_j,Te_j\rangle$ for all $j \in \mathbb N$.\qedhere
\end{itemize}
\end{proof}

\begin{remark}
\begin{enumerate}
\item The above proof shows that for any upper triangular (not necessarily compact) operator
$T \in \mathcal B(\mathcal H)$ all ``diagonal entries'' $\langle e_j,Te_j\rangle$ are eigenvalues
of $T$. This in general is false for lower triangular operators. However, if $T$ is compact and 
$\langle e_j,Te_j\rangle$ is non-zero, then it is also true for lower triangular operators as
seen above.\vspace{4pt}
\item To see what happens to Theorem \ref{lemma_12} if we waive the compactness of the operator $T$, consider the left shift on $\ell_2(\mathbb N)$ which is obviously upper triangular with respect to the standard basis of $\ell_2(\mathbb N)$. The diagonal entries are all zero, however the point spectrum of the left shift coincides with the interior of the unit disk so ``the'' diagonal elements are neither dense in the whole spectrum nor in the point spectrum.
\end{enumerate}
\end{remark}

\subsection{Examples}\label{appendix_examples}
\begin{example}\label{ex_1}
Consider the set
$E:=\lbrace C\in\mathcal B^1(\mathcal H)\,|\,\nu_1(C)\leq 1\rbrace\subset \mathcal B^1(\mathcal H)$
and define $C_n=\langle e_{n+1},\cdot\rangle e_{n+1}$, where $(e_n)_{n\in\mathbb N}$ is some orthonormal 
basis of $\mathcal H$. Obviously, $C_n \in E$ as $\nu_1(C_n)=1$. Moreover, let $\Pi_n$ be the 
corresponding orthogonal projections as in \eqref{eq:defi_block} and set $T:=\operatorname{id}_{\mathcal H}$
and $S_n=\Pi_n$ for all $n\in\mathbb N$. Then, by Lemma \ref{lemma_proj_strong_conv} (a), the projections
$\Pi_n$  converge strongly to $\operatorname{id}_{\mathcal H}$ but
\begin{align*}
\sup_{C\in E}|\operatorname{tr}(CS_n^\dagger TS_n-C)|=\sup_{C\in E}|\operatorname{tr}(C\Pi_n-C)|
\geq |\operatorname{tr}(C_n\Pi_n - C_n)|=1
\end{align*}
as $C_n\Pi_n = 0$. Hence, $\lim_{n\to\infty}\sup_{C\in E}|\operatorname{tr}(CS_n^\dagger TS_n-CS^\dagger TS)| \geq 1$,
i.e. $\operatorname{tr}\big((\cdot)S_n^\dagger TS_n\big)_{n \in \mathbb N}$ does not converges uniformly
to $\operatorname{tr}\big((\cdot)S^\dagger TS\big)$ on $E$.
\end{example}

\begin{example}\label{ex_2}
Let $(e_n)_{n\in\mathbb N}$ of $\mathcal H$ be an orthonormal basis of $\mathcal H$ 
and choose $C := \langle e_1,\cdot\rangle e_1$ and $T := \operatorname{id}_{\mathcal H}$. Then, for
the corresponding block approximations, one has $C_n = C$ and $T_n = \Pi_n$ for all $n \in \mathbb N$,
where $\Pi_n$ denotes the orthogonal projection onto $\operatorname{span}\{e_1, \dots, e_n\}$.
Therefore, we conclude
\begin{align*}
W_{C_n}(T_n) = \{\operatorname{tr}(C_nU^\dagger T_nU)\,|\, U \in \mathcal B(\mathcal H)\;\text{unitary}\}
= \{\langle x,\Pi_nx \rangle\,|\, \Vert x \Vert = 1\} = [0,1]
\end{align*}
and thus $1 = \overline{W_C(T)} \subsetneq \lim_{n\to\infty}\overline{W_{C_n}(T_n)} = [0,1]$.
\end{example}
\begin{example}\label{ex_3}
Let $(a_n)_{n\in\mathbb N} = (a_n')_{n\in\mathbb N} :=( \frac12,\frac14,\frac18,\ldots)$,
$(b_n)_{n\in\mathbb N}:=(1,1,1,\ldots)$ and $(b_n')_{n\in\mathbb N}:=(0,1,1,1,\ldots)$. Then, by the terminology
of Lemma \ref{Lemma_5b} (a), one readily verifies $A = \lbrace 1\rbrace$ and
$A'=\lbrace 1-\frac{1}{2^n}\,|\,n\in\mathbb N\rbrace$ and thus $\overline{A} \subsetneq \overline{A'}$.
\end{example}

\subsection{Alternate Proof of the Second Inclusion of Theorem \ref{theorem_4}}\label{app_sec_incl}    
Here, we present an alternative proof of the second inclusion of Theorem \ref{theorem_4} which is more oriented along the lines of the original proof \cite[Corollary 2.4]{article_sunder}. For this purpose, we need doubly stochastic operators. 
An operator $S\in\mathcal B(\mathcal H)$ is called \emph{doubly stochastic} with respect to
the orthonormal basis $(e_n)_{n\in\mathbb N}$ if the following conditions hold:
\begin{itemize}
\item
$\langle e_i,Se_j\rangle\geq 0$ for all $i,j\in\mathbb N$.\vspace{4pt}
\item
$\sum_{i=1}^\infty \langle e_i,Se_j\rangle=1$ for all $j\in\mathbb N$.\vspace{4pt}
\item
$\sum_{j=1}^\infty \langle e_i,Se_j\rangle=1$ for all $i\in\mathbb N$.
\end{itemize}
Furthermore, set
\begin{align*}
\mathcal D(\mathcal H):=\lbrace S\in\mathcal B(\mathcal H)\,|\,S\text{ is doubly stochastic w.r.t. } (e_n)_{n\in\mathbb N}\rbrace\,.
\end{align*}
Although $\mathcal D(\mathcal H)$ is not invariant under unitary conjugations (as simple finite
dimensional examples show) and does in fact depend on $(e_n)_{n\in\mathbb N}$, we avoid to express
this explicitly for simplicity of notation. The set of doubly stochastic operators can be
characterized via permutations as follows. For any permutation $\sigma:\mathbb N\to\mathbb N$
define $U_\sigma\in \mathcal B(\mathcal H)$ by
\begin{align*}
U_\sigma:=\sum_{n=1}^\infty\langle e_n,\cdot\rangle e_{\sigma(n)} \in \mathcal B(\mathcal H)\,.
\end{align*} 
This leads to
\begin{align}\label{eq:6}
\mathcal D(\mathcal H) 
= \overline{\operatorname{conv}(\lbrace U_\sigma\,|\,\sigma:\mathbb N\to\mathbb N\text{ is permutation}\rbrace)}^{\hspace{1pt} w}\,,
\end{align}
where the closure is taken with respect to the weak operator topology on $\mathcal B(\mathcal H)$,
cf.~\cite{article_kendall}.

\begin{proof}[Alternate proof of Theorem \ref{theorem_4} -- second inclusion]
Since both sets, the $C$-spectrum and the $C$-numerical range of $T$, are obviously invariant 
under unitary conjugation, we can assume w.l.o.g.~that $C$ and $T$ can be diagonalized with 
respect to the same orthonormal basis. Now let $w \in W_{C}(T)$. Then there exists unitary 
$U\in\mathcal B(\mathcal H)$ with $w = \operatorname{tr}(CU^\dagger TU)$. As $C=\sum_{i=1}^\infty \gamma_i\langle e_i,\cdot\rangle e_i$ and $T=\sum_{j=1}^\infty\tau_i\langle e_i,\cdot\rangle e_i$, a straightforward computation yields 
\begin{align*}
w = \operatorname{tr}(CU^\dagger TU) 
= \sum_{i,j=1}^\infty\gamma_i\tau_j|\langle e_j, Ue_i \rangle|^2\,.
\end{align*}
Next, we define
\begin{align*}
S := \sum_{i,j=1}^\infty  |\langle e_j,Ue_i\rangle|^2\langle e_i,\cdot\rangle e_j \in\mathcal B(\mathcal H)\,.
\end{align*}
It follows
\begin{align*}
\|Sx\|^2 & = \Big\|\sum_{i,j=1}^\infty  |\langle e_j,Ue_i\rangle|^2\langle e_i,x\rangle e_j\Big\|^2
=\sum_{j=1}^\infty\Big(\sum_{i=1}^\infty  |\langle e_j,Ue_i\rangle|^2|\langle e_i,x\rangle|\Big)^2\\
& \leq \sum_{j=1}^\infty\Big(\sum_{i=1}^\infty |\langle e_j,Ue_i\rangle|^2 \sum_{i=1}^\infty|\langle e_j,Ue_i\rangle|^2|\langle e_i,x\rangle|^2\Big)
= \sum_{j=1}^\infty \sum_{i=1}^\infty|\langle e_j,Ue_i\rangle|^2 |\langle e_i,x\rangle|^2\\
& = \sum_{i=1}^\infty|\langle e_i,x\rangle|^2 \sum_{j=1}^\infty|\langle e_j,Ue_i\rangle|^2 = \|x\|^2\,,
\end{align*}
where the estimate above results from the Cauchy-Schwarz inequality. This shows that
$S$ is indeed bounded. Then the unitarity of $U$ together with
$\langle e_j, S e_i \rangle = |\langle e_j,Ue_i\rangle|^2$ for all $i,j\in\mathbb N$ implies
$S \in\mathcal D(\mathcal H)$ and
\begin{align}\label{eq:S}
w = \sum_{j,i=1}^\infty \gamma_i\tau_j|\langle e_j,Ue_i\rangle|^2
= \sum_{j,i=1}^\infty \gamma_i\tau_j\langle e_j,Se_i\rangle\,.
\end{align}
Moreover, the following estimate shows that the right-hand side of \eqref{eq:S} converges absolutely and,
therefore, the order of summation can be interchanged:
\begin{align*}
\sum_{i,j=1}^\infty|\gamma_i| \, |\tau_j| \, |\langle e_j,Ue_i\rangle|^2
\leq \Vert T\Vert \sum_{i=1}^\infty|\gamma_i| \sum_{j=1}^\infty|\langle e_j,Ue_i\rangle|^2=\Vert T\Vert\nu_1(C)<\infty
\end{align*}
Now let $\varepsilon > 0$ be given. Then there exists $N \in \mathbb N$ such that 
\begin{align}
\sum_{i=N+1}^\infty\sum_{j=1}^\infty|\gamma_i| \, |\tau_j| \, |\langle e_j,Ue_i\rangle|^2
&\leq \Vert T\Vert \sum_{i=N+1}^\infty|\gamma_i|
< \frac{\varepsilon}{5}\,,\label{est_1}\\
\sum_{j=N+1}^\infty\sum_{i=1}^\infty|\gamma_i| \, |\tau_j| \, |\langle e_j,Ue_i\rangle|^2
&\leq \nu_1(C) \, \max_{j \geq N+1}|\tau_j| < \frac{\varepsilon}{5}\,,\label{est_2}
\end{align}
and
\begin{align}
\sum_{i=N+1}^\infty |\gamma_i||\tau_{\sigma(i)}| &\leq\Vert T\Vert\sum_{i=N+1}^\infty |\gamma_i|<\frac{\varepsilon}{5} \,, \label{est_3}\\
\sum_{j= N+1}^\infty|\gamma_{\sigma^{-1}(j)}||\tau_{j}|& \leq
\nu_1(C) \max_{j \geq N+1}|\tau_j| < \frac{\varepsilon}{5}\,.\label{est_4}
\end{align}
Note that \eqref{est_3} and \eqref{est_4} are uniform in $\sigma$, i.e.~$N \in \mathbb N$
can be chosen such that \eqref{est_3} and \eqref{est_4} hold for all permutations 
$\sigma: \mathbb N \to \mathbb N$.
Moreover, (\ref{eq:6}) guarantees the existence of finitely many permutations $\sigma_1, \dots ,\sigma_L$ 
and positive scalars $\alpha_1, \dots, \alpha_L > 0$ with $\sum_{k=1}^L\alpha_k = 1$ such that
\begin{align}\label{D_H_estimate}
\Big|\langle e_j,Se_i \rangle - \Big\langle e_j, \sum_{k=1}^L\alpha_kU_{\sigma_k}e_i \Big\rangle\Big| 
< \frac{\varepsilon}{5 \kappa N^2}
\end{align}
for all $i,j\in \lbrace 1,\ldots, N\rbrace$ and $\kappa := 1 + \max_{i,j\in \lbrace 1,\ldots, N\rbrace}|\gamma_i \tau_j|$. Now due to the
estimates \eqref{est_1} -- \eqref{D_H_estimate} we obtain
\begin{align*}
w = \sum_{i,j=1}^\infty\gamma_i\tau_j \langle e_j, Se_i \rangle
& = \sum_{i,j=1}^N \gamma_i\tau_j \langle e_j, Se_i \rangle + \Delta_1\\
& = \sum_{i,j=1}^N \sum_{k=1}^L \gamma_i\tau_j\alpha_k \langle e_j, U_{\sigma_k}e_i \rangle + \Delta_1 + \Delta_2\\
& = \sum_{k=1}^L \alpha_k \sum_{i,j=1}^N  \gamma_i\tau_j \langle e_j, e_{\sigma_k(i)} \rangle 
+ \Delta_1 + \Delta_2\\
& = \sum_{k=1}^L \alpha_k \sum_{i=1}^\infty  \gamma_i\tau_{\sigma_k(i)} 
+ \Delta_1 + \Delta_2 + \Delta_3
\end{align*}
with
\begin{align*}
|\Delta_1| & \leq  \sum_{i=N+1}^\infty\sum_{j=1}^\infty|\gamma_i| \, |\tau_j| \, |\langle e_j,Ue_i\rangle|^2
+ \sum_{j=N+1}^\infty\sum_{i=1}^\infty|\gamma_i| \, |\tau_j| \, |\langle e_j,Ue_i\rangle|^2
< \frac{2\varepsilon}{5}\,,\\
|\Delta_2| & \leq  \kappa \sum_{i,j=1}^N \Big|\langle e_j,Se_i \rangle - \Big\langle e_j, \sum_{k=1}^L\alpha_kU_{\sigma_k}e_i \Big\rangle\Big| 
< \frac{\varepsilon}{5}\,,\\
|\Delta_3| &  \leq \sum_{k=1}^L \alpha_k 
\Big| \sum_{i,j=1}^N\gamma_i\tau_j\delta(j,\sigma_k(i))\mp \sum_{i=1}^\infty\sum_{j=1}^N \gamma_i\tau_j\delta(j,\sigma_k(i)) - \sum_{i=1}^\infty\gamma_i\tau_{\sigma(i)} \Big|   \\
&\leq \sum_{k=1}^L \alpha_k \Big( \sum_{i=N+1}^\infty |\gamma_i||\tau_{\sigma_k(i)}|+  \sum_{j=N+1}^\infty|\gamma_{\sigma^{-1}(j)}||\tau_{j}|\Big) < \frac{2\varepsilon}{5}\,.
\end{align*}
Hence, it follows
\begin{align*}
\Big|w - \sum_{k=1}^L \alpha_k \sum_{i=1}^\infty  \gamma_i\tau_{\sigma_k(i)}\Big| < \varepsilon
\end{align*}
and thus $w \in \overline{\operatorname{conv}(P_C(T))}$. Finally, as the convex hull of a compact subset of $\mathbb R^n$ is compact, one has
\begin{align*}
\overline{\operatorname{conv}(P_C(T))}\subseteq \overline{\operatorname{conv}(\overline{P_C(T)})}=\operatorname{conv}(\overline{P_C(T)})\subseteq \overline{\operatorname{conv}(P_C(T))}
\end{align*}
where the last inclusion can be seen easily.
\end{proof}

\bigskip

\textbf{Acknowledgements.} The authors are grateful for valuable and constructive comments by Chi-Kwong Li, Thomas Schulte-Herbr\"uggen and the anonymous referee during the preparation of this manuscript. This work was supported in part by the Excellence Network of Bavaria (ENB) through ExQM.
\bibliographystyle{tfnlm}

\bibliography{the_C_numerical_range_in_infinite_dimensions_arxiv}
\end{document}